\documentclass[final,1p,times]{elsarticle}

\journal{Applied Numerical Mathematics}

\usepackage{amsmath,amssymb,amsthm,mathtools,graphicx,microtype}
\usepackage{bm,pgfplots}
\pgfplotsset{compat=1.17}
\usepackage{array,hyperref,enumitem}
\usepackage{booktabs}
\usepackage{orcidlink}

\newcolumntype{M}[1]{>{\centering\arraybackslash}m{#1}}

\theoremstyle{plain}
\newtheorem{theorem}{Theorem}[section]
\newtheorem{lemma}[theorem]{Lemma}

\theoremstyle{definition}
\newtheorem{definition}[theorem]{Definition}

\theoremstyle{remark}
\newtheorem{remark}[theorem]{Remark}

\newtheorem{assumption}[theorem]{Assumption}
\newtheorem{problem}[theorem]{Problem}

\newcommand{\dx}{\,\mathrm{d}x}
\newcommand{\dt}{\,\mathrm{d}t}

\newcommand{\Div}{\operatorname{div}}

\renewcommand{\rho}{\varrho}

\definecolor{myblue}{RGB}{31,119,180}
\definecolor{myorange}{RGB}{255,127,14}
\definecolor{mygreen}{RGB}{44,160,44}

\definecolor{color0}{HTML}{4E79A7}
\definecolor{color1}{HTML}{F28E2B}
\definecolor{color2}{HTML}{E15759}
\definecolor{color3}{HTML}{76B7B2}
\definecolor{color4}{HTML}{59A14F}
\definecolor{color5}{HTML}{EDC948}
\definecolor{color6}{HTML}{B07AA1}
\definecolor{color7}{HTML}{FF9DA7}
\definecolor{color8}{HTML}{9C755F}
\definecolor{color9}{HTML}{BAB0AC}

\begin{document}

\begin{frontmatter}

\title{Structure-preserving discretization and fingering dynamics of a
Cahn--Hilliard model for traction-driven digit morphogenesis}

\author[uvienna]{Marvin Fritz\orcidlink{0000-0002-8360-7371}}
% optionally add:
% \ead{marvin.fritz@...}
% \cortext[cor1]{Corresponding author}
% \author[uvienna]{Marvin Fritz\corref{cor1}}

\affiliation[uvienna]{organization={Faculty of Mathematics, University of Vienna},
            city={Vienna},
            country={Austria}}

\begin{abstract}
We study a Cahn--Hilliard equation with anisotropic traction flux arising as a
reduced continuum model of mechanically biased cell interactions in
digit-forming organoids. For a regularized problem with strictly positive
bounded mobility, we introduce a mixed finite element discretization based on
an implicit--explicit treatment of the chemical potential. We prove existence
of discrete solutions, establish exact mass conservation and a discrete energy
inequality, and show convergence of the fully discrete approximations to a weak
solution of the regularized problem. Numerical experiments illustrate the
resulting dynamics and show the transition from classical coarsening to
traction-induced fingering and protrusive growth. The computational study is
complemented by mass and energy diagnostics, an energy-balance residual,
fingering-onset and protrusion-count diagnostics, and a manufactured-solution
convergence study.
\end{abstract}


\end{frontmatter}

\section{Introduction}
We study the Cahn--Hilliard-type equation
\begin{align*}
\partial_t\rho
&=\Div(\rho\nabla q)+\varepsilon^{-1}\eta^2 \Div(\rho D(\nu)\nabla\rho),\\
q
&=\varepsilon^{-1}W'(\rho)-\varepsilon\Delta\rho,
\end{align*}
posed on bounded domains with no-flux boundary conditions, where \(D(\nu)\) is a prescribed bounded symmetric nonnegative tensor field. The model was derived in \cite{Tsutsumi2025} as a reduced continuum description of mechanically biased cell interactions in digit-forming organoids. It combines degenerate Cahn--Hilliard diffusion with an anisotropic traction-driven flux and is capable of producing protrusive, finger-like morphologies.

The continuum equation is analytically challenging because the mobility \(m(\rho)=\rho\) degenerates at vacuum and the additional traction term is not generated by the free energy. In our work \cite{fritz2026global}, the existence of global weak solutions for the one-sided degenerate problem was obtained by combining energy estimates with a mobility-matched entropy method. The present paper has the following purpose: we focus on the numerical approximation of a regularized version of the model and on the resulting pattern-forming dynamics.

In fact, we introduce a mixed finite element discretization with an implicit--explicit treatment of the chemical potential based on a convex--concave splitting of the bulk potential. The scheme is designed to retain key structural features of the continuum model at the discrete level. In particular, it preserves mass exactly and satisfies a discrete energy inequality in which the anisotropic traction term appears as a controlled perturbation of the Cahn--Hilliard dissipation mechanism.

The numerical analysis of Cahn--Hilliard equations has a long history. Classical finite element convergence results include the fully discrete analysis for logarithmic free energies by Copetti and Elliott \cite{CopettiElliott1992} and the mixed finite element error analysis of Feng and Prohl \cite{FengProhl2004}. For variable mobilities, important references are the finite element approximation of Barrett and Blowey \cite{barrett_blowey_1999} and the later works on degenerate Cahn--Hilliard systems \cite{barrett1999finite,BarrettBloweyGarcke2001}. Related structure-preserving and application-driven discretizations for Cahn--Hilliard-type models include, for example, the work on elasticity-coupled Cahn--Hilliard dynamics \cite{GarckeRumpfWeikard2001}, the variable-mobility scheme \cite{kim2007numerical}, the ternary degenerate-mobility method \cite{kim2009numerical}, the coupled tumor growth analysis \cite{brunk2025analysis,brunk2025structure,garcke2022numerical}, the work on variable mobility dynamics with nonlocal interactions \cite{brunk2025ohta}, the relative-energy stability approach \cite{brunk2026high}, and the degenerate finite element error analysis \cite{Agosti2018}. For broader overviews of numerical methods for Cahn--Hilliard equations and related energy-based models, see also the review articles \cite{brunk2026reviewthermodynamicstructuresstructurepreserving,TierraGuillenGonzalez2015}.

Against this background, the present setting is distinguished by the anisotropic higher-order traction flux. Unlike the standard mobility-weighted Cahn--Hilliard equation, the additional
traction flux is not generated by the variational derivative of the
Cahn--Hilliard energy. Hence the system is not a gradient flow of
\eqref{eq:E} with respect to the mobility-induced metric, and the analysis
requires a balance between the Cahn--Hilliard dissipation and the additional
traction transport. Our first objective is therefore structural: we prove existence of solutions to the fully discrete problem, establish exact mass conservation, derive a discrete energy inequality, and show convergence, for fixed mobility regularization, of the fully discrete approximations as \(h,\tau\to0\) to a weak solution of the regularized continuous problem. This provides a mathematically consistent approximation framework for the anisotropic traction model.

Our second objective is computational. Using the proposed scheme, we
investigate how the traction strength \(\eta\) and the prescribed direction
field \(\nu\) influence the morphology of the evolving phase field. The
simulations show a clear transition from the nearly isotropic coarsening regime
of the classical Cahn--Hilliard equation to a traction-dominated regime in
which the diffuse interface destabilizes and develops elongated protrusions.
In particular, radial traction fields produce robust finger-like structures
whose onset time, number, and wavelength depend systematically on the traction
strength. We also report structure-preservation diagnostics, including a
one-step energy-balance residual, and a manufactured-solution convergence
study. 

The biological motivation is rooted in classical and modern ideas on self-organized pattern formation. Beginning with Turing's seminal work on morphogenesis \cite{Turing1952}, reaction--diffusion and related symmetry-breaking mechanisms have become central in developmental pattern formation; see, for example, the review of Kondo and Miura \cite{KondoMiura2010}. In the specific context of limb and digit formation, Turing-type ideas have been revisited and supported in increasingly detailed developmental studies; see, for example, \cite{Raspopovic2014,Sheth2012}. The reduced organoid model considered here is different in mechanism, since it couples diffuse-interface dynamics with mechanically biased transport, but it addresses the same broader question of how robust protrusive and digit-like patterns can emerge from local interaction rules.

The rest of the paper is organized as follows.
In Section~\ref{sec:model} we recall the reduced continuum model and introduce
the regularized problem used in the numerical analysis.
Section~\ref{sec:numerics} introduces the fully discrete finite element scheme,
proves existence of discrete solutions, establishes its basic
structure-preserving properties, and shows convergence to a weak solution of
the regularized problem.
Section~\ref{sec:sims} presents numerical experiments, including
structure-preservation diagnostics, traction-induced fingering, an empirical
onset and protrusion-count study, and a manufactured-solution convergence
benchmark.
Finally, Section~\ref{sec:conclusion} summarizes the results and discusses
open problems.

\section{Model and regularization}\label{sec:model}

We briefly recall the reduced continuum model introduced in \cite{Tsutsumi2025} and then state the regularized problem used in the numerical analysis.

Let \(\Omega\subset\mathbb{R}^d\) (\(d\in\{2,3\}\)) be a bounded domain with outer unit normal \(n\), and let \(T>0\). The reduced continuum model for the cell density \(\rho=\rho(x,t)\) is
\begin{align}
\partial_t\rho
&=\Div\bigl(\rho \nabla q\bigr)
+\kappa \Div\bigl(\rho D(\nu)\nabla\rho\bigr),
\label{eq:PDE}
\\
q
&=\varepsilon^{-1}W'(\rho)-\varepsilon\Delta\rho,
\label{eq:q2}
\end{align}
where \(\kappa=\varepsilon^{-1}\eta^2\), \(W\) is a double-well-type bulk potential, and \(D(\nu)\) is a prescribed bounded symmetric nonnegative tensor field encoding anisotropic traction along a preferred direction field \(\nu\). We impose the no-flux boundary conditions
\begin{equation}\label{eq:BC}
\nabla\rho\cdot n = 0,
\qquad
\bigl(\rho\nabla q+\kappa \rho D(\nu)\nabla\rho\bigr)\cdot n=0
\qquad\text{on }\partial\Omega\times(0,T),
\end{equation}
together with the initial condition
\[
\rho(\cdot,0)=\rho_0
\qquad\text{in }\Omega.
\]
The model arises as a reduced one-field description of mechanically biased cell interactions in digit-forming organoids \cite{Tsutsumi2025}. The Cahn--Hilliard part describes phase-separating density rearrangement, whereas the additional traction term biases transport along preferred directions and can induce protrusive, finger-like morphologies.

In the present paper we do not discretize the degenerate mobility \(m(\rho)=\rho\) directly. Instead, we consider a regularized problem with a mobility
\[
m\in C^2(\mathbb{R}),
\qquad
0<m_1\le m(r)\le m_2
\qquad\forall r\in\mathbb{R},
\]
which is strictly positive, globally bounded, and smooth. The regularized problem reads
\begin{align}
\partial_t\rho
&=\Div\bigl(m(\rho)\nabla \mu\bigr)
+\kappa \Div\bigl(m(\rho)D(\nu)\nabla\rho\bigr),
\label{eq:reg_model_rho}
\\
\mu
&=\varepsilon^{-1}W'(\rho)-\varepsilon\Delta\rho,
\label{eq:reg_model_mu}
\end{align}
subject to
\begin{equation}\label{eq:reg_model_bc}
\nabla\rho\cdot n = 0,
\qquad
\bigl(m(\rho)\nabla\mu+\kappa m(\rho)D(\nu)\nabla\rho\bigr)\cdot n = 0
\qquad\text{on }\partial\Omega\times(0,T).
\end{equation}
This system retains the conservative structure of the original equation while removing the degeneracy of the mobility, which is the setting used throughout the numerical analysis below.

The associated free energy is
\begin{equation}\label{eq:E}
E(\rho):=\int_\Omega\Bigl(\varepsilon^{-1}W(\rho)+\frac{\varepsilon}{2}|\nabla\rho|^2\Bigr)\dx.
\end{equation}
Moreover, under the no-flux boundary conditions, both the original and the regularized problems conserve the total mass \(\int_\Omega \rho(t)\dx\).
For the time discretization we shall use a convex--concave splitting of the bulk potential \(W\); this is introduced together with the fully discrete scheme in Section~\ref{sec:numerics}.

\begin{remark}[Relation to energy-based anisotropic variants]
The mobility-weighted Cahn--Hilliard equation
\[
\partial_t\rho=\Div(m(\rho)\nabla\mu),
\qquad
\mu=\frac{\delta E}{\delta\rho},
\]
is a generalized gradient flow of \(E\) with respect to the
mobility-induced Onsager metric. The additional traction flux in
\eqref{eq:reg_model_rho},
\[
\kappa\Div\bigl(m(\rho)D(\nu)\nabla\rho\bigr),
\]
is different in nature: it is imposed as an extra conservative transport term
and is not, in general, generated by the variational derivative of the
Cahn--Hilliard energy \eqref{eq:E}.

To compare with a genuinely energy-based modification, suppose for simplicity
that the anisotropy is represented by a scalar coefficient
\(a=a(\nu(x))\). Adding the quadratic term
\[
\frac{\kappa}{2}\int_\Omega a(x)\rho^2\dx
\]
to the energy would give the chemical potential
\[
\mu_{\rm eb}
=
\varepsilon^{-1}W'(\rho)-\varepsilon\Delta\rho+\kappa a(x)\rho
\]
and hence the additional flux
\[
\kappa\Div\bigl(m(\rho)\nabla(a\rho)\bigr)
=
\kappa\Div\bigl(m(\rho)a\nabla\rho\bigr)
+
\kappa\Div\bigl(m(\rho)\rho\nabla a\bigr).
\]
Thus even in the scalar case the energy-based model differs from the traction
model by the drift term
\(\kappa\Div(m(\rho)\rho\nabla a)\), unless \(a\) is spatially constant. For
the tensor-valued anisotropy \(D(\nu)\) used in this paper, such as
\(D=\nu\otimes\nu\), the second-order term
\(\Div(m(\rho)D(\nu)\nabla\rho)\) is therefore best understood as a prescribed
anisotropic traction flux rather than as the gradient flow of a local scalar
energy.
\end{remark}

\section{Finite element discretization}\label{sec:numerics}

\begin{assumption}\label{ass:num_model}
Assume:
\begin{enumerate}\itemsep0em
\item \(\Omega\subset\mathbb R^d\) (\(d\in\{2,3\}\)) is bounded with \(C^{1,1}\) boundary, \(T>0\), \(\varepsilon>0\), \(\eta\ge0\), and \(\kappa:=\varepsilon^{-1}\eta^2\).
\item The prescribed anisotropy tensor satisfies
\[
D(\nu)\in C([0,T];L^\infty(\Omega;\mathbb R^{d\times d})),
\qquad
D(\nu)=D(\nu)^\top,
\qquad
0\le D(\nu)\le \Lambda I
\quad\text{a.e. in }\Omega_T .
\]
\item \(W\in C^2(\mathbb R)\), and there exist constants \(c_0,c_1,C_W>0\) and an exponent
\[
2<p<\infty \quad\text{if } d=2,
\qquad
2<p<4 \quad\text{if } d=3,
\]
such that
\[
W(s)\ge -c_0,
\qquad
W''(s)\ge -c_1
\qquad\forall s\in\mathbb R,
\]
and
\[
|W'(s)|\le C_W(1+|s|^{p-1}),
\qquad
|W''(s)|\le C_W(1+|s|^{p-2})
\qquad\forall s\in\mathbb R.
\]
\end{enumerate}
\end{assumption}

\begin{remark}[Coverage of the standard double-well potential]
In dimension \(d=2\), Assumption~\ref{ass:num_model}(3) allows arbitrary
polynomial growth and therefore covers the standard quartic potential
\[
W(s)=\frac14 s^2(1-s)^2
\]
used in the numerical experiments of Section~\ref{sec:sims}. In dimension
\(d=3\), the restriction \(p<4\) excludes the borderline quartic case from
the convergence statement of Theorem~\ref{thm:conv_scheme}. This restriction
enters only through the passage to the limit in the nonlinear term
\(W_1'(\rho_{h,\tau})\), where we use strong convergence in
\(L^2(0,T;L^{2(p-1)}(\Omega))\) and the Sobolev embedding
\(H^1(\Omega)\hookrightarrow L^6(\Omega)\). Treating the three-dimensional
quartic case would require an additional argument, for instance a compactness
or monotonicity argument that does not rely on this strict subcritical
embedding. The structure-preserving results
Theorems~\ref{thm:disc_structure_num}--\ref{thm:exist_step_num} are not
affected by this restriction.
\end{remark}

Assumption~\ref{ass:num_model}(3) implies that \(W\) is \(c_1\)-semiconvex.
Equivalently, \(W\) admits the convex--concave splitting
\[
W=W_1+W_2
\]
with
\[
W_1(s):=W(s)+\frac{c_1}{2}s^2 \ \text{convex},
\qquad
W_2(s):=-\frac{c_1}{2}s^2 \ \text{concave}.
\]
In particular,
\[
W_1''(s)=W''(s)+c_1\ge0,
\qquad
W_2''(s)\equiv -c_1,
\]
and therefore
\begin{equation}\label{eq:splitting_identity_num}
W_1'(s)+W_2'(r)=W'(s)+c_1(s-r)\qquad(\forall r,s\in\mathbb R).
\end{equation}

We now present the mixed finite element discretization used in the simulations of Section~\ref{sec:sims}, and analyze its structure-preserving properties and convergence. Throughout, we consider a regularized problem in which the mobility is strictly positive, globally bounded, and smoothed near the truncation points. To keep the notation light, we suppress the regularization parameters. Accordingly, \(m\) denotes a strictly positive and globally bounded \(C^2\)-mobility satisfying
\begin{equation}\label{eq:mob_assump_num}
m\in C^2(\mathbb R),\qquad
0<m_1\le m(r)\le m_2,\qquad
\|m'\|_{L^\infty(\mathbb R)}\le m_3,\qquad
\|m''\|_{L^\infty(\mathbb R)}\le m_4.
\end{equation}
For the numerical experiments below, the regularization is applied only at the level of the mobility, while the bulk potential is taken to be the original double-well potential \(W\). Let \(\tau>0\), \(t^n=n\tau\) (\(n=0,\dots,N\), \(N\tau=T\)), and denote
\[
D^n:=D(\nu(\cdot,t^n)).
\]
Furthermore, let \(\{\mathcal T_h\}_{h>0}\) be a shape-regular and quasi-uniform
family of triangulations of \(\Omega\), and let
\[
V_h:=\{v_h\in C^0(\overline\Omega):\ v_h|_K\ \text{affine for all }K\in\mathcal T_h\}\subset H^1(\Omega)
\]
be the conforming \(P_1\) space.
To fix the additive constant in the chemical potential, we set
\[
V_h^0:=\{v_h\in V_h:\ (v_h,1)=0\}.
\]
We approximate \(\rho\) in \(V_h\) and \(q\) in \(V_h^0\), i.e. we use a mixed \(P_1\)--\(P_1\) scheme with mean-zero constraint for the chemical potential.
At the discrete level we work with the mean-zero representative \(q_h^n\in V_h^0\) of the chemical potential, which is the natural unknown for the saddle-point structure of the mixed scheme. The full chemical potential \(\mu_h^n\) is reconstructed in \eqref{eq:mu_full_discrete} below and is used only in the convergence analysis of Section~\ref{subsec:conv}.

\begin{problem}[Fully discrete step]\label{prob:fullydisc_simple}
Given \(\rho_h^{n-1}\in V_h\), define the explicit mobility \(m^{n-1}:=m(\rho_h^{n-1})\) with \(m^{n-1}\ge m_1>0\).
Find \((\rho_h^n,q_h^n)\in V_h\times V_h^0\) such that for all \((\varphi,\psi)\in V_h\times V_h^0\),
\begin{align}
\frac{1}{\tau} \bigl(\rho_h^n-\rho_h^{n-1},\varphi\bigr)
&+ \bigl(m^{n-1}\nabla q_h^n,\nabla\varphi\bigr)
+ \kappa \bigl(m^{n-1}D^n\nabla\rho_h^n,\nabla\varphi\bigr)=0,
\label{eq:scheme1_imex_num}
\\
(q_h^n,\psi)
&=\varepsilon^{-1}\bigl(W_1'(\rho_h^n)+W_2'(\rho_h^{n-1}),\psi\bigr)
+\varepsilon(\nabla\rho_h^n,\nabla\psi).
\label{eq:scheme2_imex_num}
\end{align}
\end{problem}

The first equation is a finite element weak form of a conservative balance.
Indeed, with the numerical fluxes
\[
J_{1,h}^n:=m^{n-1}\nabla q_h^n,
\qquad
J_{2,h}^n:=\kappa m^{n-1}D^n\nabla\rho_h^n,
\]
\eqref{eq:scheme1_imex_num} can be written as
\[
\frac{1}{\tau}
\left(\rho_h^n-\rho_h^{n-1},\varphi\right)
=
-(J_{1,h}^n+J_{2,h}^n,\nabla\varphi)
\qquad\forall \varphi\in V_h.
\]
Thus conservation is imposed in the usual weak finite element sense, by testing
the divergence form against scalar test functions. In particular, choosing
\(\varphi\equiv1\) gives exact conservation of the total mass.

\begin{theorem}[Structure-preserving properties]\label{thm:disc_structure_num}
Let \((\rho_h^n,q_h^n)\) solve \eqref{eq:scheme1_imex_num}--\eqref{eq:scheme2_imex_num}. Then:
\begin{enumerate}\itemsep0em
\item \textbf{Mass conservation.}
\[
(\rho_h^n,1)=(\rho_h^{n-1},1)\qquad \text{for all }n\ge 1.
\]

\item \textbf{Discrete energy inequality.}
We write \(E_h:=E|_{V_h}\) for the restriction of the continuum energy
\eqref{eq:E} to the finite element space, i.e.
\[
E_h(\rho_h):=\int_\Omega
\Big(\varepsilon^{-1}W(\rho_h)+\frac{\varepsilon}{2}|\nabla\rho_h|^2\Big)\dx
\qquad(\rho_h\in V_h).
\]
Then for every \(n\ge1\),
\begin{equation}\label{eq:disc_energy_ineq_num}
\begin{aligned}
E_h(\rho_h^n)
&+\frac{\varepsilon}{2}\|\nabla(\rho_h^n-\rho_h^{n-1})\|_{L^2(\Omega)}^2
+\frac{\tau}{2}\int_\Omega m^{n-1}|\nabla q_h^n|^2\dx
\\
&\le
E_h(\rho_h^{n-1})
+ C(\Lambda)\kappa^2 \tau\int_\Omega m^{n-1}|\nabla\rho_h^n|^2\dx .
\end{aligned}
\end{equation}
In particular, if \(\kappa=0\) then the scheme is unconditionally energy stable:
\[
E_h(\rho_h^n)
+\frac{\varepsilon}{2}\|\nabla(\rho_h^n-\rho_h^{n-1})\|_{L^2(\Omega)}^2
+\tau\int_\Omega m^{n-1}|\nabla q_h^n|^2\dx
\le E_h(\rho_h^{n-1})\qquad(n\ge1).
\]
\end{enumerate}
\end{theorem}

\begin{proof}
Mass conservation follows from choosing \(\varphi\equiv1\) in \eqref{eq:scheme1_imex_num}.
For \eqref{eq:disc_energy_ineq_num}, test \eqref{eq:scheme1_imex_num} with \(\varphi=q_h^n\).
Moreover, by mass conservation,
\[
(\rho_h^n-\rho_h^{n-1},1)=0,
\]
hence \(\rho_h^n-\rho_h^{n-1}\in V_h^0\), and we may test
\eqref{eq:scheme2_imex_num} with \(\psi=\rho_h^n-\rho_h^{n-1}\).
Using the polarization identity for the gradient term and the convexity/concavity inequalities
\[
W_1'(b)(b-a)\ge W_1(b)-W_1(a),\qquad
W_2'(a)(b-a)\ge W_2(b)-W_2(a),
\]
we obtain
\[
\begin{aligned}
&\frac{1}{\tau}\bigl(E_h(\rho_h^n)-E_h(\rho_h^{n-1})\bigr)
+\frac{\varepsilon}{2\tau}\|\nabla(\rho_h^n-\rho_h^{n-1})\|_{L^2}^2
+\int_\Omega m^{n-1}|\nabla q_h^n|^2\dx
\\&\le -\kappa\int_\Omega m^{n-1}D^n\nabla\rho_h^n\cdot\nabla q_h^n\dx.
\end{aligned}
\]
Finally, \(0\le D^n\le \Lambda I\) and Young's inequality yield
\[
\kappa\Big|\int_\Omega m^{n-1}D^n\nabla\rho_h^n\cdot\nabla q_h^n\dx\Big|
\le \frac12\int_\Omega m^{n-1}|\nabla q_h^n|^2\dx
+ C(\Lambda)\kappa^2\int_\Omega m^{n-1}|\nabla\rho_h^n|^2\dx,
\]
Multiplying by \(\tau\) gives \eqref{eq:disc_energy_ineq_num}. If
\(\kappa=0\), the traction term is absent and the Young estimate is not used;
therefore the full dissipation coefficient \(\tau\) remains, which gives the
stated unconditional energy stability.
\end{proof}

\begin{theorem}[Existence of a discrete solution at time step \(n\)]\label{thm:exist_step_num}
Assume \eqref{eq:mob_assump_num} and let \(W=W_1+W_2\) be the convex--concave splitting described above.
Then there exists a constant \(\tau_e>0\), depending only on \(\varepsilon\), \(\kappa\), \(m_2\), and \(\Lambda\)
(and in particular independent of \(h\) and \(n\)), such that for every
\(0<\tau\le \tau_e\) and every \(\rho_h^{n-1}\in V_h\), there exists at least one pair
\((\rho_h^n,q_h^n)\in V_h\times V_h^0\) satisfying
\eqref{eq:scheme1_imex_num}--\eqref{eq:scheme2_imex_num}.
\end{theorem}

\begin{proof}
We argue by a direct coupled Brouwer-type argument.

\smallskip
\noindent\textit{Step 1: Reduction to a mean-zero increment.}
Let \(y_h^n:=\rho_h^n-\rho_h^{n-1}\).
Testing \eqref{eq:scheme1_imex_num} with \(\varphi\equiv 1\) shows that any solution satisfies
\[
(y_h^n,1)=0.
\]
Hence the unknown increment belongs to the mean-zero space \(V_h^0\), and conversely every
\(y\in V_h^0\) defines a mass-conserving candidate \(\rho=\rho_h^{n-1}+y\).
Therefore it is enough to solve for
\[
(y,q)\in X_h:=V_h^0\times V_h^0,
\qquad \rho:=\rho_h^{n-1}+y.
\]
We equip \(X_h\) with the Hilbert structure
\[
((y,q),(\phi,\psi))_{X_h}:=(\nabla y,\nabla\phi)+(\nabla q,\nabla\psi).
\]
Since all functions in \(V_h^0\) have zero mean, Poincar\'e's inequality implies that this scalar product
induces a norm equivalent to the \(H^1(\Omega)\times H^1(\Omega)\) norm on \(X_h\).

\smallskip
\noindent\textit{Step 2: Residual operator.}
For \(z=(y,q)\in X_h\), define \(\mathcal F_h(z)\in X_h'\) by
\begin{equation}\label{eq:Fh_def}
\begin{aligned}
\langle \mathcal F_h(y,q),(\phi,\psi)\rangle
:=\;&
\frac1\tau (y,\phi)
+\bigl(m^{n-1}\nabla q,\nabla\phi\bigr)
+\kappa\bigl(m^{n-1}D^n\nabla\rho,\nabla\phi\bigr)
\\ &\;
+\,(q,\psi)
-\varepsilon^{-1}\bigl(W_1'(\rho)+W_2'(\rho_h^{n-1}),\psi\bigr)
-\varepsilon\bigl(\nabla\rho,\nabla\psi\bigr)
\end{aligned}
\end{equation}
for all \((\phi,\psi)\in X_h\).
Since \(X_h\) is finite-dimensional and all nonlinearities are continuous, \(\mathcal F_h\) is continuous.
By construction, \(\mathcal F_h(y,q)=0\) is equivalent to
\eqref{eq:scheme1_imex_num}--\eqref{eq:scheme2_imex_num} on \(V_h^0\).
Because the constant test function in \eqref{eq:scheme1_imex_num} is already incorporated through \((y,1)=0\),
this is equivalent to the full scheme on \(V_h\times V_h^0\).

\smallskip
\noindent\textit{Step 3: Energy-testing transform.}
Define the linear isomorphism
\(\mathcal T_h:X_h\to X_h\)
by
\[
\mathcal T_h(\phi,\psi):=(-\tau\psi,\phi).
\]
For every \(z\in X_h\), let \(\mathcal G_h(z)\in X_h\) be the unique element satisfying
\begin{equation}\label{eq:Gh_def}
(\mathcal G_h(z),w)_{X_h}
=
-\langle \mathcal F_h(z),\mathcal T_h w\rangle
\qquad\forall w\in X_h.
\end{equation}
The existence and uniqueness of \(\mathcal G_h(z)\) follow from the Riesz representation theorem on the finite-dimensional Hilbert space \(X_h\).
The map \(\mathcal G_h:X_h\to X_h\) is continuous.
Moreover, since \(\mathcal T_h\) is invertible, we have
\[
\mathcal G_h(z)=0
\quad\Longleftrightarrow\quad
\langle \mathcal F_h(z),\widetilde w\rangle=0
\ \ \forall \widetilde w\in X_h
\quad\Longleftrightarrow\quad
\mathcal F_h(z)=0.
\]
Hence zeros of \(\mathcal G_h\) are exactly the discrete solutions.
We now estimate \((\mathcal G_h(z),z)_{X_h}\). By \eqref{eq:Gh_def} and \eqref{eq:Fh_def},
for \(z=(y,q)\in X_h\) with \(\rho=\rho_h^{n-1}+y\),
\begin{equation}\label{eq:Gzz_start}
\begin{aligned}
(\mathcal G_h(z),z)_{X_h}
&=
-\langle \mathcal F_h(y,q),(-\tau q,y)\rangle
\\
&=
\varepsilon^{-1}\bigl(W_1'(\rho)+W_2'(\rho_h^{n-1}),\,\rho-\rho_h^{n-1}\bigr)
+\varepsilon\bigl(\nabla\rho,\nabla(\rho-\rho_h^{n-1})\bigr)
\\ &\quad
+\tau\int_\Omega m^{n-1}|\nabla q|^2\dx
+\tau\kappa\int_\Omega m^{n-1}D^n\nabla\rho\cdot\nabla q\dx.
\end{aligned}
\end{equation}

\smallskip
\noindent\textit{Step 4: Lower bound by the discrete energy.}
Using the convexity of \(W_1\) and the concavity of \(W_2\), we obtain
\begin{equation}\label{eq:split_est_exist}
\varepsilon^{-1}\bigl(W_1'(\rho)+W_2'(\rho_h^{n-1}),\rho-\rho_h^{n-1}\bigr)
\ge
\varepsilon^{-1}\int_\Omega \bigl(W(\rho)-W(\rho_h^{n-1})\bigr)\dx.
\end{equation}
Moreover, the polarization identity yields
\begin{equation}\label{eq:grad_pol_exist}
\varepsilon\bigl(\nabla\rho,\nabla(\rho-\rho_h^{n-1})\bigr)
=
\frac{\varepsilon}{2}\|\nabla\rho\|_{L^2(\Omega)}^2
-\frac{\varepsilon}{2}\|\nabla\rho_h^{n-1}\|_{L^2(\Omega)}^2
+\frac{\varepsilon}{2}\|\nabla y\|_{L^2(\Omega)}^2.
\end{equation}
Combining \eqref{eq:Gzz_start}--\eqref{eq:grad_pol_exist} gives
\begin{equation}\label{eq:Gzz_energy}
\begin{aligned}
(\mathcal G_h(z),z)_{X_h}
\ge\;&
E_h(\rho)-E_h(\rho_h^{n-1})
+\frac{\varepsilon}{2}\|\nabla y\|_{L^2(\Omega)}^2
+\tau\int_\Omega m^{n-1}|\nabla q|^2\dx
\\ &\;
+\tau\kappa\int_\Omega m^{n-1}D^n\nabla\rho\cdot\nabla q\dx.
\end{aligned}
\end{equation}
For the traction term, we use \(0\le D^n\le \Lambda I\) a.e.\ and Young's inequality:
\[
\begin{aligned}
\tau\kappa\Big|\int_\Omega m^{n-1}D^n\nabla\rho\cdot\nabla q\dx\Big|
&\le
\tau\kappa\Lambda\int_\Omega m^{n-1}|\nabla\rho||\nabla q|\dx
\\
&\le
\frac{\tau}{2}\int_\Omega m^{n-1}|\nabla q|^2\dx
+
C(\Lambda)\tau\kappa^2\int_\Omega m^{n-1}|\nabla\rho|^2\dx .
\end{aligned}
\]
Since \(m^{n-1}\le m_2\), we further have
\[
\int_\Omega m^{n-1}|\nabla\rho|^2\dx
\le
m_2\|\nabla\rho\|_{L^2(\Omega)}^2
\le
2m_2\|\nabla y\|_{L^2(\Omega)}^2
+
2m_2\|\nabla\rho_h^{n-1}\|_{L^2(\Omega)}^2.
\]
Therefore, from \eqref{eq:Gzz_energy},
\[
\begin{aligned}
(\mathcal G_h(z),z)_{X_h}
\ge\;&
E_h(\rho)-E_h(\rho_h^{n-1})
+
\Bigl(\frac{\varepsilon}{2}-C_0\tau\kappa^2\Bigr)\|\nabla y\|_{L^2(\Omega)}^2
\\ &\;
+\frac{\tau}{2}\int_\Omega m^{n-1}|\nabla q|^2\dx
-
C_0\tau\kappa^2\|\nabla\rho_h^{n-1}\|_{L^2(\Omega)}^2,
\end{aligned}
\]
with a constant \(C_0>0\) depending only on \(m_2\) and \(\Lambda\).
By Assumption~\ref{ass:num_model}(3) we have \(W(s)\ge -c_0\), hence \(E_h(\rho)\ge -C\) uniformly in \(\rho\).
Moreover, \(m^{n-1}\ge m_1>0\), and thus
\begin{equation}\label{eq:Gzz_coercive}
(\mathcal G_h(z),z)_{X_h}
\ge
\Bigl(\frac{\varepsilon}{2}-C_0\tau\kappa^2\Bigr)\|\nabla y\|_{L^2(\Omega)}^2
+\frac{\tau m_1}{2}\|\nabla q\|_{L^2(\Omega)}^2
-C_n,
\end{equation}
where
\[
C_n:=E_h(\rho_h^{n-1})+C+C_0\tau\kappa^2\|\nabla\rho_h^{n-1}\|_{L^2(\Omega)}^2.
\]
Choose now
\[
\tau_e:=
\begin{cases}
\displaystyle \frac{\varepsilon}{4C_0\kappa^2}, & \kappa>0,\\[0.6em]
+\infty, & \kappa=0.
\end{cases}
\]
If \(\kappa=0\), the traction term vanishes and the coercivity estimate holds for every \(\tau>0\), so no step restriction is needed.
Then, for every \(0<\tau\le \tau_e\), \eqref{eq:Gzz_coercive} implies
\[
(\mathcal G_h(z),z)_{X_h}
\ge
\frac{\varepsilon}{4}\|\nabla y\|_{L^2(\Omega)}^2
+\frac{\tau m_1}{2}\|\nabla q\|_{L^2(\Omega)}^2
-C_n.
\]
By Poincar\'e's inequality on \(V_h^0\), the \(X_h\)-norm is equivalent to
\[
\|(y,q)\|_{X_h}^2
\sim
\|\nabla y\|_{L^2(\Omega)}^2+\|\nabla q\|_{L^2(\Omega)}^2.
\]
Hence
\(
(\mathcal G_h(z),z)_{X_h}\to\infty\) as \(\|z\|_{X_h}\to\infty.
\)

\smallskip
\noindent\textit{Step 5: Application of Brouwer's theorem.}
Choose \(R>0\) so large that
\[
(\mathcal G_h(z),z)_{X_h}>0
\qquad\forall z\in X_h \text{ with }\|z\|_{X_h}=R.
\]
Since \(X_h\) is finite-dimensional and \(\mathcal G_h\) is continuous, a standard
corollary of Brouwer's fixed-point theorem (see, e.g., \cite[Chapter~9.1]{evans2022partial})
implies that \(\mathcal G_h\) has a zero in the closed ball \(\overline B_R(0)\subset X_h\).
Hence there exists \((y_h^n,q_h^n)\in X_h\) such that
\(
\mathcal G_h(y_h^n,q_h^n)=0,
\)
and therefore also \(\mathcal F_h(y_h^n,q_h^n)=0\).
Finally, set
\(
\rho_h^n:=\rho_h^{n-1}+y_h^n.
\)
Then \((\rho_h^n,q_h^n)\in V_h\times V_h^0\) solves
\eqref{eq:scheme1_imex_num}--\eqref{eq:scheme2_imex_num}.
\end{proof}

\begin{remark}[Uniqueness]
For \(\kappa=0\), uniqueness of the fully implicit convex-splitting step is standard under mild additional assumptions ensuring strict monotonicity of the discrete chemical-potential operator.
For \(\kappa>0\), uniqueness of the fully coupled step typically requires a mild time-step restriction depending on \(\kappa\), \(\Lambda\), and \(\|m^{n-1}\|_{L^\infty(\Omega)}\).
Since existence is sufficient for the purposes of the present work, we do not pursue a full uniqueness analysis here.
\end{remark}

\subsection{Convergence to the regularized problem}\label{subsec:conv}

In this subsection we prove convergence of the fully discrete scheme
\eqref{eq:scheme1_imex_num}--\eqref{eq:scheme2_imex_num} to a weak solution of the regularized problem. Throughout this subsection, Assumption~\ref{ass:num_model} and \eqref{eq:mob_assump_num} are in force.

Since the discrete variable \(q_h^n\in V_h^0\) is the mean-zero representative of the
chemical potential, we introduce the full discrete chemical potential by
\begin{equation}\label{eq:mu_full_discrete}
\bar\mu_h^n
:=
\frac{\varepsilon^{-1}}{|\Omega|}
\int_\Omega \Bigl(W_1'(\rho_h^n)+W_2'(\rho_h^{n-1})\Bigr)\dx,
\qquad
\mu_h^n:=q_h^n+\bar\mu_h^n.
\end{equation}
By mass conservation (Theorem~\ref{thm:disc_structure_num}) and the splitting identity \eqref{eq:splitting_identity_num}, the mean value \(\bar\mu_h^n\) admits the equivalent representation
\[
\bar\mu_h^n
= \frac{\varepsilon^{-1}}{|\Omega|}\int_\Omega W'(\rho_h^n)\dx
+ \frac{\varepsilon^{-1} c_1}{|\Omega|}\int_\Omega(\rho_h^n-\rho_h^{n-1})\dx
= \frac{\varepsilon^{-1}}{|\Omega|}\int_\Omega W'(\rho_h^n)\dx,
\]
which is the natural discrete analogue of the continuum mean of \(\mu\).
Then \(\nabla\mu_h^n=\nabla q_h^n\), and \eqref{eq:scheme2_imex_num} is equivalent to
\begin{equation}\label{eq:scheme2_mu_full}
(\mu_h^n,\psi_h)
=
\varepsilon^{-1}\bigl(W_1'(\rho_h^n)+W_2'(\rho_h^{n-1}),\psi_h\bigr)
+\varepsilon(\nabla\rho_h^n,\nabla\psi_h)
\qquad\forall \psi_h\in V_h .
\end{equation}

For \(t\in(t_{n-1},t_n]\), define the standard time reconstructions
\begin{align}
\rho_{h,\tau}(t):=\rho_h^n, \quad \rho_{h,\tau}^-(t):=\rho_h^{n-1}, \quad q_{h,\tau}(t):=q_h^n, \quad \mu_{h,\tau}(t):=\mu_h^n, \quad D_\tau(t):=D^n,
\label{eq:reconstructions_pc}
\end{align}
and the piecewise affine reconstruction
\begin{equation}\label{eq:reconstruction_affine}
\widehat\rho_{h,\tau}(t)
:=
\frac{t-t_{n-1}}{\tau}\rho_h^n+\frac{t_n-t}{\tau}\rho_h^{n-1}
\qquad (t\in[t_{n-1},t_n]).
\end{equation}
Then
\begin{equation}\label{eq:reconstruction_dt}
\partial_t\widehat\rho_{h,\tau}(t)=\frac{\rho_h^n-\rho_h^{n-1}}{\tau}
\qquad\text{for }t\in(t_{n-1},t_n].
\end{equation}

We now state the weak formulation of the regularized problem.

\begin{definition}[Weak solution of the regularized problem]\label{def:weak_reg}
A pair \((\rho,\mu)\) is called a weak solution of the regularized problem if
\[
\rho\in L^\infty(0,T;H^1(\Omega))\cap H^1(0,T;H^{-1}(\Omega)),
\qquad
\mu\in L^2(0,T;H^1(\Omega)),
\]
\(W'(\rho)\in L^2(\Omega_T)\), \(\rho(0)=\rho_0\) in \(L^2(\Omega)\), and
\begin{equation}\label{eq:weak_reg_mass}
\int_0^T \langle \partial_t\rho,\varphi\rangle\,dt
+\int_0^T\!\!\int_\Omega m(\rho)\nabla\mu\cdot\nabla\varphi \dx\dt
+\kappa\int_0^T\!\!\int_\Omega m(\rho)D(\nu)\nabla\rho\cdot\nabla\varphi \dx\dt
=0
\end{equation}
for all \(\varphi\in L^2(0,T;H^1(\Omega))\), and
\begin{equation}\label{eq:weak_reg_mu}
\int_0^T\!\!\int_\Omega \mu\,\psi \dx\dt
=
\varepsilon^{-1}\int_0^T\!\!\int_\Omega W'(\rho)\psi \dx\dt
+\varepsilon\int_0^T\!\!\int_\Omega \nabla\rho\cdot\nabla\psi \dx\dt
\end{equation}
for all \(\psi\in L^2(0,T;H^1(\Omega))\).
\end{definition}

We denote by \(P_h:L^2(\Omega)\to V_h\) the \(L^2\)-orthogonal projection,
\[
(P_h v,\chi_h)=(v,\chi_h)\qquad\forall \chi_h\in V_h .
\]
Since the triangulations are assumed to be shape-regular and quasi-uniform,
\(P_h\) is stable in \(H^1(\Omega)\); see, for instance,
\cite{BramblePasciakSteinbach2002}. Hence
\begin{equation}\label{eq:projection_properties}
\|P_h v\|_{H^1(\Omega)}\le C\|v\|_{H^1(\Omega)}
\qquad\forall v\in H^1(\Omega),
\end{equation}
with a constant independent of \(h\). Moreover,
\begin{equation}\label{eq:projection_convergence}
P_h v\to v \quad\text{strongly in }H^1(\Omega)
\qquad\text{for every }v\in H^1(\Omega).
\end{equation}
%Consequently, also
%\[
%P_h v\to v
%\quad\text{strongly in }L^2(0,T;H^1(\Omega))
%\qquad\text{for every }v\in L^2(0,T;H^1(\Omega)).
%\]

The first step is a uniform discrete a priori estimate.

\begin{lemma}[Uniform bounds]\label{lem:uniform_bounds_conv}
Assume Assumption~\ref{ass:num_model}, \eqref{eq:mob_assump_num}, and let
\((\rho_h^n,q_h^n)_{n=0}^N\) be a discrete solution of
\eqref{eq:scheme1_imex_num}--\eqref{eq:scheme2_imex_num}, with \(\tau\le\tau_*\) where $\tau_*$ depends only on $\varepsilon$, $\kappa$, $m_1$, $m_2$, $\Lambda$ and $\Omega$.
Assume further that
\begin{equation}\label{eq:init_conv_ass}
\rho_h^0\to \rho_0 \quad\text{strongly in }H^1(\Omega),
\qquad
\sup_{h>0} E_h(\rho_h^0)<\infty.
\end{equation}
Then there exists a constant \(C>0\), independent of \(h\) and \(\tau\), such that
\begin{align}
\max_{0\le n\le N}\|\rho_h^n\|_{H^1(\Omega)} &\le C,
\label{eq:uniform_rho_H1_discrete}
\\
\tau\sum_{n=1}^N \|q_h^n\|_{H^1(\Omega)}^2 &\le C,
\label{eq:uniform_q_H1_discrete}
\\
\tau\sum_{n=1}^N \|\mu_h^n\|_{H^1(\Omega)}^2 &\le C,
\label{eq:uniform_mu_H1_discrete}
\\
\|\widehat\rho_{h,\tau}\|_{L^\infty(0,T;H^1(\Omega))} &\le C,
\label{eq:uniform_rhohat_H1}
\\
\|\mu_{h,\tau}\|_{L^2(0,T;H^1(\Omega))}+\|q_{h,\tau}\|_{L^2(0,T;H^1(\Omega))} &\le C.
\label{eq:uniform_muqtau_H1}
\end{align}
\end{lemma}

\begin{proof}
By Theorem~\ref{thm:disc_structure_num},
\[
E_h(\rho_h^n)
+\frac{\tau}{2}\int_\Omega m^{n-1}|\nabla q_h^n|^2\dx
\le
E_h(\rho_h^{n-1})
+ C(\Lambda)\kappa^2\tau\int_\Omega m^{n-1}|\nabla\rho_h^n|^2\dx .
\]
Using \(m^{n-1}\le m_2\) and the lower bound \(W(s)\ge -c_0\), we obtain
\[
E_h(\rho_h^n)\ge \frac{\varepsilon}{2}\|\nabla\rho_h^n\|_{L^2(\Omega)}^2-C,
\]
and therefore
\[
E_h(\rho_h^n)
\le
E_h(\rho_h^{n-1})+C\tau\bigl(1+E_h(\rho_h^n)\bigr).
\]
Hence, for \(\tau\le \tau_*:=\min(\tau_e,\tfrac{1}{2C})\),
\[
E_h(\rho_h^n)\le (1+2C\tau)E_h(\rho_h^{n-1})+2C\tau,
\]
and the discrete Grönwall lemma yields
\[
\max_{0\le n\le N} E_h(\rho_h^n)\le C.
\]
Since \(E_h\) controls the \(H^1\)-seminorm of \(\rho_h^n\), and mass is conserved by
Theorem~\ref{thm:disc_structure_num}, Poincar\'e's inequality implies
\eqref{eq:uniform_rho_H1_discrete}. Moreover, because \(m^{n-1}\ge m_1>0\),
\[
\tau\sum_{n=1}^N \|\nabla q_h^n\|_{L^2(\Omega)}^2\le C.
\]
Since \(q_h^n\in V_h^0\), Poincar\'e's inequality yields \eqref{eq:uniform_q_H1_discrete}.

It remains to estimate the mean value of \(\mu_h^n\). By definition,
\[
\bar\mu_h^n
=
\frac{\varepsilon^{-1}}{|\Omega|}
\int_\Omega \Bigl(W_1'(\rho_h^n)+W_2'(\rho_h^{n-1})\Bigr)\dx.
\]
By the growth assumptions on \(W\), the corresponding growth bounds for \(W_1'\) and
\(W_2'\), and the Sobolev embedding \(H^1(\Omega)\hookrightarrow L^{p-1}(\Omega)\)
(which holds under Assumption~\ref{ass:num_model}(3) for both \(d=2\) and \(d=3\)),
we infer
\[
|\bar\mu_h^n|
\le
C\Bigl(1+\|\rho_h^n\|_{L^{p-1}(\Omega)}^{p-1}
+\|\rho_h^{n-1}\|_{L^{p-1}(\Omega)}^{p-1}\Bigr)
\le C.
\]
Therefore
\[
\tau\sum_{n=1}^N |\bar\mu_h^n|^2\le C.
\]
Since \(\mu_h^n=q_h^n+\bar\mu_h^n\), this implies
\eqref{eq:uniform_mu_H1_discrete}. Finally,
\eqref{eq:uniform_rhohat_H1} and \eqref{eq:uniform_muqtau_H1} follow directly from the
definitions of the time reconstructions.
\end{proof}

The next lemma yields the compactness required for the limit process.

\begin{lemma}[Time derivative bound and compactness]\label{lem:compactness_conv}
Under the assumptions of Lemma~\ref{lem:uniform_bounds_conv}, there exists a constant
\(C>0\), independent of \(h\) and \(\tau\), such that
\begin{equation}\label{eq:dt_rhohat_bound}
\|\partial_t\widehat\rho_{h,\tau}\|_{L^2(0,T;H^{-1}(\Omega))}\le C.
\end{equation}
Consequently, there exist a subsequence (not relabelled) and a function
\[
\rho\in L^\infty(0,T;H^1(\Omega))\cap H^1(0,T;H^{-1}(\Omega))
\]
such that
\begin{align}
\widehat\rho_{h,\tau}&\rightharpoonup^\ast \rho
&&\text{in }L^\infty(0,T;H^1(\Omega)),
\label{eq:rhohat_weakstar_conv}
\\
\partial_t\widehat\rho_{h,\tau}&\rightharpoonup \partial_t\rho
&&\text{in }L^2(0,T;H^{-1}(\Omega)),
\label{eq:dt_rhohat_weak_conv}
\\
\widehat\rho_{h,\tau}&\to \rho
&&\text{strongly in }C([0,T];L^2(\Omega)).
\label{eq:rhohat_strong_C_L2}
\end{align}
For the piecewise-constant reconstructions, we additionally have
\begin{align}
\rho_{h,\tau}&\rightharpoonup^\ast \rho,
&
\rho_{h,\tau}^-&\rightharpoonup^\ast \rho
&&\text{in }L^\infty(0,T;H^1(\Omega)),
\label{eq:rho_pc_weakstar_conv}
\\
\rho_{h,\tau}&\to \rho,
&
\rho_{h,\tau}^-&\to \rho
&&\text{strongly in }L^2(\Omega_T),
\label{eq:rho_pc_strong_L2}
\end{align}
and, under Assumption~\ref{ass:num_model}(3),
\begin{align}
\rho_{h,\tau}\to \rho,\qquad
\rho_{h,\tau}^-\to \rho
\qquad\text{strongly in }L^2(0,T;L^{2(p-1)}(\Omega)).
\label{eq:rho_pc_strong_L2Lp}
\end{align}
\end{lemma}

\begin{proof}
Let \(\varphi\in H^1(\Omega)\). Since
\(\partial_t\widehat\rho_{h,\tau}(t)\in V_h\) for
\(t\in(t_{n-1},t_n]\), the \(L^2\)-orthogonality of \(P_h\) gives
\[
\left\langle \partial_t\widehat\rho_{h,\tau}(t),\varphi\right\rangle
=
\bigl(\partial_t\widehat\rho_{h,\tau}(t),\varphi\bigr)
=
\bigl(\partial_t\widehat\rho_{h,\tau}(t),P_h\varphi\bigr).
\]
Using \eqref{eq:scheme1_imex_num} with the test function
\(P_h\varphi\in V_h\), we therefore obtain, for
\(t\in(t_{n-1},t_n]\),
\[
\begin{aligned}
\left\langle \partial_t\widehat\rho_{h,\tau}(t),\varphi\right\rangle
&=
-\bigl(m^{n-1}\nabla q_h^n,\nabla P_h\varphi\bigr)
-\kappa\bigl(m^{n-1}D^n\nabla\rho_h^n,\nabla P_h\varphi\bigr).
\end{aligned}
\]
Using \eqref{eq:projection_properties}, \(m^{n-1}\le m_2\),
\(0\le D^n\le \Lambda I\), and Cauchy--Schwarz, we infer
\[
\left|
\left\langle \partial_t\widehat\rho_{h,\tau}(t),\varphi\right\rangle
\right|
\le
C\Bigl(
\|\nabla q_h^n\|_{L^2(\Omega)}
+
\|\nabla\rho_h^n\|_{L^2(\Omega)}
\Bigr)
\|\varphi\|_{H^1(\Omega)}.
\]
Hence
\[
\|\partial_t\widehat\rho_{h,\tau}(t)\|_{H^{-1}(\Omega)}
\le
C\Bigl(
\|\nabla q_h^n\|_{L^2(\Omega)}
+
\|\nabla\rho_h^n\|_{L^2(\Omega)}
\Bigr),
\]
and \eqref{eq:dt_rhohat_bound} follows from
\eqref{eq:uniform_rho_H1_discrete}--\eqref{eq:uniform_q_H1_discrete}.

Now \eqref{eq:uniform_rhohat_H1} and \eqref{eq:dt_rhohat_bound} imply, by the2 Aubin--Lions compactness theorem \cite[Lemma 7.10]{roubicek2005nonlinear},
\[
\widehat\rho_{h,\tau}\to \rho
\qquad\text{strongly in }C([0,T];L^2(\Omega)),
\]
after extraction of a subsequence. This yields
\eqref{eq:rhohat_weakstar_conv}--\eqref{eq:rhohat_strong_C_L2}. 

It remains to compare the piecewise constant and affine reconstructions. For
\(t\in(t_{n-1},t_n]\),
\[
\rho_{h,\tau}(t)-\widehat\rho_{h,\tau}(t)
=
\frac{t_n-t}{\tau}\bigl(\rho_h^n-\rho_h^{n-1}\bigr),
\]
and therefore, using \eqref{eq:reconstruction_dt},
\[
\|\rho_{h,\tau}-\widehat\rho_{h,\tau}\|_{L^2(0,T;H^{-1}(\Omega))}
\le
\tau \|\partial_t\widehat\rho_{h,\tau}\|_{L^2(0,T;H^{-1}(\Omega))}
\le C\tau \to 0.
\]
Likewise,
\[
\|\rho_{h,\tau}^- -\widehat\rho_{h,\tau}\|_{L^2(0,T;H^{-1}(\Omega))}
\le C\tau \to 0.
\]
Since \(\rho_{h,\tau}\), \(\rho_{h,\tau}^-\), and \(\widehat\rho_{h,\tau}\) are uniformly bounded in
\(L^\infty(0,T;H^1(\Omega))\), the interpolation inequality
\[
\|v\|_{L^2(\Omega)}^2
\le
\|v\|_{H^{-1}(\Omega)}\|v\|_{H^1(\Omega)}
\]
yields
\[
\|\rho_{h,\tau}-\widehat\rho_{h,\tau}\|_{L^2(\Omega_T)}\to0,
\qquad
\|\rho_{h,\tau}^- -\widehat\rho_{h,\tau}\|_{L^2(\Omega_T)}\to0.
\]
Together with \eqref{eq:rhohat_strong_C_L2}, this proves \eqref{eq:rho_pc_strong_L2}. Since \(\rho_{h,\tau}\), \(\rho_{h,\tau}^-\), and \(\widehat\rho_{h,\tau}\) are uniformly bounded in \(L^\infty(0,T;H^1(\Omega))\), the convergences \eqref{eq:rho_pc_weakstar_conv} follow, after extraction of subsequences, from Banach--Alaoglu and the identification of the limit through \eqref{eq:rho_pc_strong_L2}.

Finally, \eqref{eq:rho_pc_strong_L2Lp} follows from \eqref{eq:rho_pc_strong_L2},
the uniform \(L^\infty(0,T;H^1(\Omega))\)-bound, and interpolation in space.
Indeed, if \(d=3\), then Assumption~\ref{ass:num_model}(3) implies \(2(p-1)<6\), and if \(d=2\),
the embedding \(H^1(\Omega)\hookrightarrow L^r(\Omega)\) holds for every finite \(r\).
\end{proof}

We can now pass to the limit in the discrete scheme.

\begin{theorem}[Convergence of the fully discrete scheme]\label{thm:conv_scheme}
Assume Assumption~\ref{ass:num_model}, \eqref{eq:mob_assump_num}, and
\eqref{eq:init_conv_ass}. Let
\((\rho_h^n,q_h^n)_{n=0}^N\) be a family of discrete solutions of
\eqref{eq:scheme1_imex_num}--\eqref{eq:scheme2_imex_num}, and define
\(\mu_h^n\), \(\rho_{h,\tau}\), \(\rho_{h,\tau}^-\), \(\widehat\rho_{h,\tau}\),
\(\mu_{h,\tau}\), and \(D_\tau\) by
\eqref{eq:mu_full_discrete}--\eqref{eq:reconstruction_affine}. Then, along a subsequence
as \(h,\tau\to0\), there exist
\[
\rho\in L^\infty(0,T;H^1(\Omega))\cap H^1(0,T;H^{-1}(\Omega)),
\qquad
\mu\in L^2(0,T;H^1(\Omega)),
\]
such that \(\rho\) has the compactness properties stated in
Lemma~\ref{lem:compactness_conv},
\[
\mu_{h,\tau}\rightharpoonup \mu
\qquad\text{weakly in }L^2(0,T;H^1(\Omega)),
\]
and the pair \((\rho,\mu)\) is a weak solution of the regularized problem
in the sense of Definition~\ref{def:weak_reg}. Moreover,
\[
\rho(0)=\rho_0 \quad\text{in }L^2(\Omega),
\qquad
\int_\Omega \rho(t)\dx=\int_\Omega \rho_0\dx
\quad\text{for all }t\in[0,T].
\]
\end{theorem}

\begin{proof}
By \eqref{eq:uniform_muqtau_H1}, there exists \(\mu\in L^2(0,T;H^1(\Omega))\) such that,
after extraction,
\[
\mu_{h,\tau}\rightharpoonup \mu
\qquad\text{weakly in }L^2(0,T;H^1(\Omega)).
\]
Since \(m\) is globally Lipschitz and bounded, \eqref{eq:rho_pc_strong_L2} implies
\begin{equation}\label{eq:m_strong_conv}
m(\rho_{h,\tau}^-)\to m(\rho)
\qquad\text{strongly in }L^2(\Omega_T).
\end{equation}
Furthermore, by Assumption~\ref{ass:num_model}(2) and the definition of \(D_\tau\),
\begin{equation}\label{eq:Dtau_strong_conv}
D_\tau\to D(\nu)
\qquad\text{strongly in }L^\infty(\Omega_T).
\end{equation}

We next identify the nonlinear term in the chemical potential relation.
Since \(W_1''=W''+c_1\), Assumption~\ref{ass:num_model}(3) implies
\[
|W_1''(s)|\le C(1+|s|^{p-2})
\qquad\forall s\in\mathbb R.
\]
Hence, by the mean-value theorem,
\[
|W_1'(a)-W_1'(b)|
\le C\bigl(1+|a|^{p-2}+|b|^{p-2}\bigr)|a-b|
\qquad\forall a,b\in\mathbb R.
\]
Using Hölder's inequality, the strong convergence \eqref{eq:rho_pc_strong_L2Lp}, and the uniform bound of
\(\rho_{h,\tau}\) in \(L^\infty(0,T;L^{2(p-1)}(\Omega))\), we obtain
\[
W_1'(\rho_{h,\tau})\to W_1'(\rho)
\qquad\text{strongly in }L^2(\Omega_T).
\]
Since \(W_2'\) is linear,
\[
W_2'(\rho_{h,\tau}^-)\to W_2'(\rho)
\qquad\text{strongly in }L^2(\Omega_T).
\]
Hence
\begin{equation}\label{eq:Wstrong_conv_num}
W_1'(\rho_{h,\tau})+W_2'(\rho_{h,\tau}^-)\to W'(\rho)
\qquad\text{strongly in }L^2(\Omega_T).
\end{equation}

Let \(\varphi\in L^2(0,T;H^1(\Omega))\), and set
\(\varphi_h:=P_h\varphi\). By \eqref{eq:projection_convergence},
\[
\varphi_h\to\varphi
\qquad\text{strongly in }L^2(0,T;H^1(\Omega)).
\]
Using the piecewise affine reconstruction of the time derivative, the discrete scheme gives
\[
\int_0^T \langle \partial_t\widehat\rho_{h,\tau},\varphi_h\rangle\,dt
+\int_0^T\!\!\int_\Omega m(\rho_{h,\tau}^-)\nabla\mu_{h,\tau}\cdot\nabla\varphi_h \dx\dt
+\kappa\int_0^T\!\!\int_\Omega m(\rho_{h,\tau}^-)D_\tau\nabla\rho_{h,\tau}\cdot\nabla\varphi_h \dx\dt
=0.
\]
The first term converges to
\[
\int_0^T \langle \partial_t\rho,\varphi\rangle\,dt
\]
by \eqref{eq:dt_rhohat_weak_conv} and \eqref{eq:projection_convergence}. The second term converges to
\[
\int_0^T\!\!\int_\Omega m(\rho)\nabla\mu\cdot\nabla\varphi \dx\dt
\]
by \eqref{eq:m_strong_conv}, the weak convergence of \(\nabla\mu_{h,\tau}\), and
\(\nabla\varphi_h\to \nabla\varphi\) strongly in \(L^2(\Omega_T)\). The traction term converges to
\[
\kappa\int_0^T\!\!\int_\Omega m(\rho)D(\nu)\nabla\rho\cdot\nabla\varphi \dx\dt
\]
by \eqref{eq:m_strong_conv}, \eqref{eq:Dtau_strong_conv}, the weak convergence
\(\nabla\rho_{h,\tau}\rightharpoonup \nabla\rho\) in \(L^2(\Omega_T)\), and
\(\nabla\varphi_h\to\nabla\varphi\) strongly in \(L^2(\Omega_T)\).
This proves \eqref{eq:weak_reg_mass}.

To obtain the weak chemical potential identity, let \(\psi\in L^2(0,T;H^1(\Omega))\), and set
\(\psi_h:=P_h\psi\). Again,
\[
\psi_h\to\psi
\qquad\text{strongly in }L^2(0,T;H^1(\Omega)).
\] By \eqref{eq:scheme2_mu_full},
\[
\int_0^T\!\!\int_\Omega \mu_{h,\tau}\psi_h \dx\dt
=
\varepsilon^{-1}\int_0^T\!\!\int_\Omega
\bigl(W_1'(\rho_{h,\tau})+W_2'(\rho_{h,\tau}^-)\bigr)\psi_h \dx\dt
+\varepsilon\int_0^T\!\!\int_\Omega \nabla\rho_{h,\tau}\cdot\nabla\psi_h \dx\dt.
\]
Passing to the limit by the weak convergence of \(\mu_{h,\tau}\), the strong convergence
\eqref{eq:Wstrong_conv_num}, the weak convergence of \(\nabla\rho_{h,\tau}\), and
\eqref{eq:projection_convergence}, we obtain \eqref{eq:weak_reg_mu}.

Since \(\widehat\rho_{h,\tau}(0)=\rho_h^0\) and \(\widehat\rho_{h,\tau}\to \rho\) strongly in
\(C([0,T];L^2(\Omega))\), while \(\rho_h^0\to \rho_0\) strongly in \(L^2(\Omega)\), it follows that
\[
\rho(0)=\rho_0 \qquad\text{in }L^2(\Omega).
\]
Finally, mass conservation holds at the discrete level:
\[
\int_\Omega \rho_h^n \dx = \int_\Omega \rho_h^0 \dx
\qquad\text{for all }n.
\]
Passing to the limit by \eqref{eq:rhohat_strong_C_L2} and \eqref{eq:init_conv_ass}, we obtain
\[
\int_\Omega \rho(t)\dx=\int_\Omega \rho_0\dx
\qquad\text{for all }t\in[0,T].
\]
This completes the proof.
\end{proof}

\begin{remark}[Scope of the convergence result]
Theorem~\ref{thm:conv_scheme} gives convergence of the fully discrete finite element
scheme to a weak solution of the \emph{regularized} anisotropic Cahn--Hilliard problem.
Passing simultaneously to the degenerate mobility limit, for instance by letting
\(m_1\downarrow0\) and \(m_2\uparrow\infty\), is substantially more delicate and is not
addressed here.
\end{remark}

% =====================================================================
%  Section 5 (or 4): Numerical experiments  --  CONSOLIDATED, DEDUPLICATED
%
%  Flow:
%    intro
%    5.1 Implementation, parameters, and diagnostics      (setup)
%    5.2 Traction-induced patterning and energy diagnostics
%        - snapshots (Fig:evolution, images)
%        - mass + energy (Fig:mass-energy-n128, pgfplots)
%        - anisotropy + dissipation (Fig:anisotropy-dissipation, pgfplots)
%        - energy residual (Fig:energy-residual, pgfplots)
%    5.3 Onset and multiplicity of protrusions            (fingering, pgfplots)
%    5.4 Manufactured-solution convergence study          (pgfplots)
%
%  Data file layout assumed:
%    Data/   : snapshot PNGs, viridis.png,
%              convergence_mms_eta05.txt, fingering_Nt.txt, fingering_peak.txt
%    Data2/  : n128_fin{eta}.txt  (from run_diagnostics_full.py),
%              columns 0:t 1:mass 2:energy 3:dissipation 4:anisotropy
%                      5:cross 6:grad_diff_sq 7:R_residual 8:mass_drift
%  Adjust the Data/ vs Data2/ prefixes to your local layout if needed.
% =====================================================================

\section{Numerical experiments}\label{sec:sims}

We illustrate traction-driven pattern formation and protrusive growth within
the regularized framework analyzed in Section~\ref{sec:numerics}, using the
scheme \eqref{eq:scheme1_imex_num}--\eqref{eq:scheme2_imex_num} implemented in
\texttt{Firedrake} \cite{rathgeber2016firedrake}. All computations in this
section are two-dimensional; in particular, the quartic double-well potential
used below is covered by Assumption~\ref{ass:num_model}(3).

The purpose of the experiments is threefold. First, we test the conservative
and energy-structured character of the discretization by monitoring mass,
energy, dissipation, and the one-step energy-balance residual associated with
Theorem~\ref{thm:disc_structure_num}. Second, we demonstrate the qualitative
transition from classical Cahn--Hilliard coarsening to traction-induced
fingering, and we quantify the onset and multiplicity of the resulting
protrusions. Third, we report a manufactured-solution convergence study for the
regularized scheme. The latter is an empirical consistency test of the
implementation; the convergence result proved in Section~\ref{subsec:conv} is
compactness-based and does not assert a quantitative rate.

\subsection{Implementation, parameters, and diagnostics}
\label{subsec:numerical_setup}

At each time step we solve
\eqref{eq:scheme1_imex_num}--\eqref{eq:scheme2_imex_num} by a damped Newton
method with line search, while the linearized systems are solved by Krylov
methods with a direct \(LU\)-preconditioner. Unless stated otherwise we work
in \(\Omega=(0,1)^2\) with a uniform triangulation of mesh size \(h=1/256\),
use \(V_h=\mathbb P_1\) elements, and set the interfacial width to
\(\varepsilon=10^{-2}\), so that the diffuse interface is resolved by
approximately five mesh cells across. We use the double-well potential
\[
W(\rho)=\tfrac14\rho^2(1-\rho)^2,
\]
the time step \(\tau=10^{-5}\), and the smooth, globally bounded mobility
\begin{equation}\label{eq:sim_mobility}
m_\ast(r):=m_-+\tfrac12\bigl(r+\sqrt{r^2+\delta^2}\bigr)
       -\tfrac12\bigl((r-M)+\sqrt{(r-M)^2+\delta^2}\bigr),
\end{equation}
with $m_-=10^{-3}$, $M=2$, $\delta=10^{-3}$,
evaluated explicitly at the previous time step,
\(m^{n-1}=m_\ast(\rho_h^{n-1})\). This is a smooth clipping of the density
mobility \(m(\rho)=\rho\) to the interval \([0,M]\) with a small positive
floor: on the physically relevant range \(\rho\in[0,1]\) one has
\(m_\ast(\rho)\approx m_-+\rho\), while \(m_\ast(r)\to m_-\) as
\(r\to-\infty\) and \(m_\ast(r)\to m_-+M\) as \(r\to+\infty\). The bounds
\(m_-\le m_\ast\le m_-+M\) hold globally on \(\mathbb R\), so that
assumption~\eqref{eq:mob_assump_num} applies directly to the implemented
mobility.

We also use the convex--concave splitting of the bulk potential introduced in
Section~\ref{sec:numerics}: with \(W=W_1+W_2\) and
\(W_1'(\rho_h^n)+W_2'(\rho_h^{n-1})=W'(\rho_h^n)+c_1(\rho_h^n-\rho_h^{n-1})\),
the chemical-potential equation \eqref{eq:scheme2_imex_num} is evaluated with
\(c_1=1\). For the quartic potential one has
\(W''(s)=3s^2-3s+\tfrac12\), whose minimum is
\(\min_s W''(s)=W''(\tfrac12)=-\tfrac14\); hence the semiconvexity bound
\(W''\ge -c_1\) holds for any \(c_1\ge\tfrac14\), and the choice \(c_1=1\)
is admissible, giving \(W_1''=W''+1\ge\tfrac34>0\). This is the scheme to
which the structure-preserving analysis of Theorem~\ref{thm:disc_structure_num}
applies.

To generate outward-biased protrusions we use the radial direction field
\(\nu\) and set \(D=\nu\otimes\nu\), where
\begin{equation}\label{Eq:Radial}
\nu(x):=\frac{x-x_0}{\sqrt{|x-x_0|^2+\sigma^2}},
\qquad
x_0=\biggl(\frac12,\frac12\biggr),
\qquad
\sigma=10^{-2}.
\end{equation}
This models an outward traction bias aligned with a coarse-grained cue pointing
from the aggregate interior toward the periphery, a natural idealization for
approximately radially symmetric organoid geometries. We initialize
\(\rho_h^0\) as a smooth disk with diffuse interface,
\[
\rho_h^0(x)=\tfrac12\left(1-\tanh\left(
\dfrac{|x-(\frac12,\frac12)|-0.28}{3\varepsilon}\right)\right).
\]

Throughout, we monitor the mass and energy,
\[
M_h(t^n):=\int_\Omega\rho_h^n\dx,
\qquad
E_h^n:=E_h(\rho_h^n),
\]
the mobility-weighted dissipation proxy and the anisotropy measure
\begin{equation}\label{eq:diagnostics_D_A}
\mathcal D_h^n:=\int_\Omega m^{n-1}|\nabla q_h^n|^2\dx,
\qquad
\mathcal A_h^n:=\int_\Omega m^{n-1}\nabla\rho_h^n\cdot D^n\nabla\rho_h^n\dx ,
\end{equation}
and the one-step energy-balance residual
\begin{equation}\label{eq:energy_residual_num}
\begin{aligned}
\mathcal R_h^n:={}&
E_h(\rho_h^n)-E_h(\rho_h^{n-1})
+\frac{\varepsilon}{2}\|\nabla(\rho_h^n-\rho_h^{n-1})\|_{L^2(\Omega)}^2
+\tau\int_\Omega m^{n-1}|\nabla q_h^n|^2\dx
\\
&\quad
+\tau\kappa\int_\Omega m^{n-1}D^n\nabla\rho_h^n\cdot\nabla q_h^n\dx .
\end{aligned}
\end{equation}
Here \(\mathcal D_h^n\) is the dissipation term in the Cahn--Hilliard part of
the energy estimate and \(\mathcal A_h^n\) measures the alignment of the
current density gradient with the prescribed traction tensor. The proof of
Theorem~\ref{thm:disc_structure_num} gives \(\mathcal R_h^n\le0\) for an exact
solution of the nonlinear algebraic system, before the subsequent use of
Young's inequality; thus \eqref{eq:energy_residual_num} is a sharper diagnostic
than the energy alone, retaining the traction cross-term.

We emphasize that the time-step restriction in the existence proof of
Theorem~\ref{thm:exist_step_num} is a sufficient coercivity condition and is
not expected to be sharp. The simulations below were run in a regime in which
the Newton solver converged robustly and the monitored structural quantities
behaved consistently with the discrete estimates.

\subsection{Traction-induced patterning and energy diagnostics}
\label{subsec:patterning_results}

We vary the traction strength \(\eta\) for fixed interfacial width
\(\varepsilon\). For \(\eta=0\) the model reduces to the usual Cahn--Hilliard
dynamics, with phase separation and coarsening; for \(\eta>0\) the anisotropic
traction flux biases transport and produces elongated structures aligned with
\(\nu\). Figure~\ref{fig:evolution} shows the evolution of the phase field for
\(\eta\in\{0,0.25,0.5,0.75,1\}\). For \(\eta=0\) the interface stays
essentially isotropic and the evolution is dominated by curvature-driven
relaxation. For \(\eta=0.25\) the traction is weak and the result is close to
the \(\eta=0\) evolution. As \(\eta\) increases, the radial bias destabilizes
the interface and produces elongated protrusions: intermediate \(\eta\) yields
a small number of persistent protrusions, while larger \(\eta\) leads to a
fingering/branching regime in which thin arms form and propagate outward.

\begin{figure}[htp!]
\centering
\includegraphics[page=1,width=.99\textwidth]{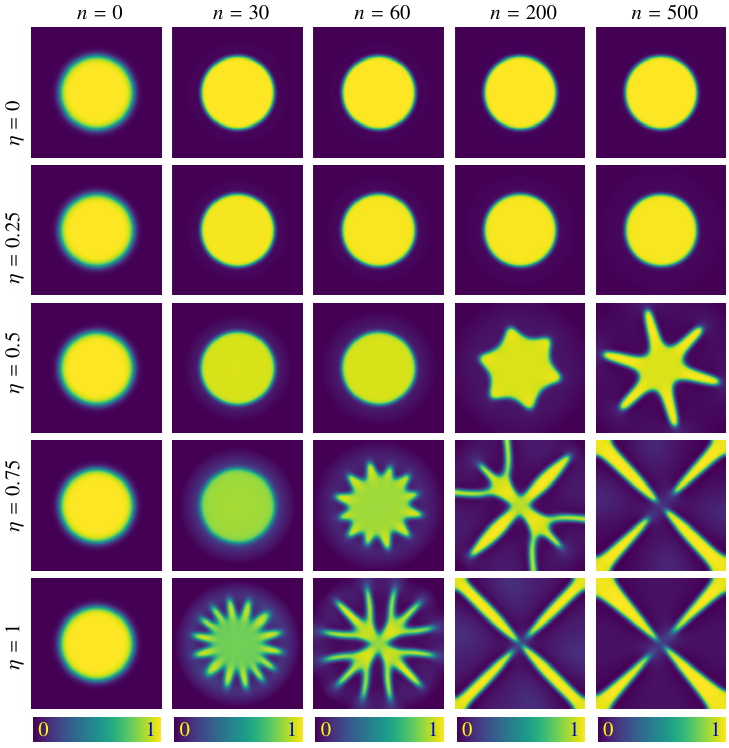}
\caption{Time snapshots of the phase field \(\rho_h^n\) under radial traction
bias. Columns show successive time steps (increasing left to right); rows
correspond to increasing traction strength \(\eta\) (top to bottom). For
\(\eta=0\) the interface remains essentially isotropic and evolves by classical
Cahn--Hilliard coarsening, whereas increasing \(\eta\) triggers a
traction-induced interfacial instability and the formation of finger-like
protrusions. The colour scale is fixed to \(\rho\in[0,1]\).}
\label{fig:evolution}
\end{figure}

Figure~\ref{fig:mass-energy-n128} reports the discrete energy and the mass
drift. The mass is conserved up to machine precision throughout, confirming
that the conservative formulation \eqref{eq:scheme1_imex_num} introduces no
artificial gain or loss. The energy behaves differently across regimes: for \(\eta=0\) it decreases monotonically, as expected from the
mobility-weighted Cahn--Hilliard gradient-flow structure associated with
\(E_h\). For larger \(\eta\), the energy can increase, since the traction flux
is not generated by the variational derivative of \(E_h\) and can feed energy
into the interface as protrusions form.

\begin{figure}[htp!]
\centering
\includegraphics[page=2,width=.47\textwidth]{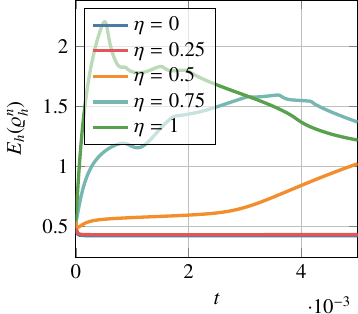}\hfill\includegraphics[page=3,width=.45\textwidth]{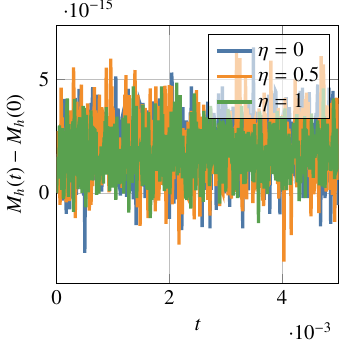}
\caption{Left: discrete free energy \(E_h(\rho_h^n)\) versus time
\(t^n=n\tau\) for \(\eta\in\{0,0.25,0.5,0.75,1\}\). Right: mass drift
\(M_h(t^n)-M_h(0)\) for \(\eta\in\{0,0.5,1\}\), confirming conservation up to
machine precision.}
\label{fig:mass-energy-n128}
\end{figure}

To resolve this transient we display, in Figure~\ref{fig:anisotropy-dissipation},
the traction anisotropy measure \(\mathcal A_h^n\) and the mobility-weighted
dissipation proxy \(\mathcal D_h^n\). Both decay overall, with transient
increases during the protrusion-formation phase. This is consistent with the
discrete energy inequality \eqref{eq:disc_energy_ineq_num}: the traction
contribution acts as a controlled perturbation of the gradient-flow dissipation
mechanism.

\begin{figure}[htp!]
\centering
\includegraphics[page=4,width=.47\textwidth]{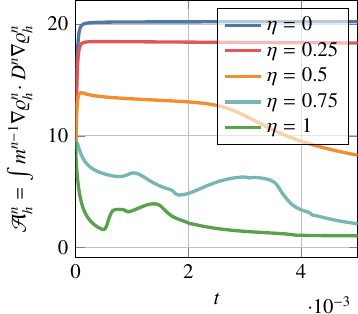}\hfill\includegraphics[page=5,width=.45\textwidth]{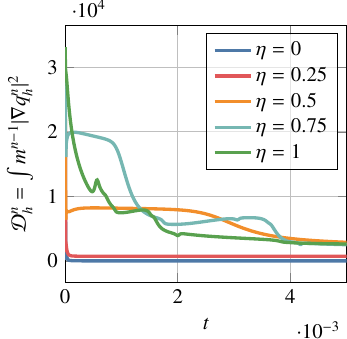}
\caption{Left: anisotropy measure
\(\mathcal A_h^n=\int_\Omega m^{n-1}\nabla\rho_h^n\cdot D^n\nabla\rho_h^n\dx\).
Right: mobility-weighted dissipation proxy
\(\mathcal D_h^n=\int_\Omega m^{n-1}|\nabla q_h^n|^2\dx\). Traction acts as a controlled perturbation of the Cahn--Hilliard
energy-dissipation mechanism.}
\label{fig:anisotropy-dissipation}
\end{figure}

Finally, Figure~\ref{fig:energy-residual} displays the one-step
energy-balance residual \eqref{eq:energy_residual_num}. In the reported runs
the residual remains nonpositive up to solver and roundoff tolerances; for the
two extreme cases we measure
\(\max_n\mathcal R_h^n\approx-5.6\times10^{-9}\) for \(\eta=0\) and
\(\max_n\mathcal R_h^n\approx-8.6\times10^{-6}\) for \(\eta=0.5\). This verifies
at the algebraic level the exact balance used in the proof of
Theorem~\ref{thm:disc_structure_num}, and it explains the contrasting energy
behaviour of Figure~\ref{fig:mass-energy-n128}: the traction cross-term in
\eqref{eq:energy_residual_num} is negative and can outweigh the dissipation, so
\(E_h^n\) may increase while the full balance remains dissipative. The residual
therefore distinguishes a controlled traction-induced energy transient from a
numerical loss of stability.

\begin{figure}[htp!]
\centering
\includegraphics[page=6,width=.55\textwidth]{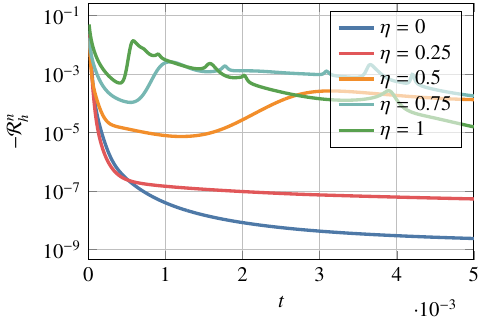}
\caption{One-step energy-balance residual \eqref{eq:energy_residual_num},
plotted as \(-\mathcal R_h^n\ge0\) on a logarithmic axis. Since
\(\mathcal R_h^n\le0\) holds throughout (Theorem~\ref{thm:disc_structure_num}),
the curves confirm the discrete inequality up to the nonlinear solver and
roundoff tolerances, for all reported traction strengths.}
\label{fig:energy-residual}
\end{figure}

\subsection{Onset and multiplicity of protrusions}
\label{subsec:fingering}

The snapshots of Figure~\ref{fig:evolution} indicate that both the onset and
the number of protrusions depend systematically on the traction strength
\(\eta\). We now quantify this dependence and complement it with a heuristic
instability argument for the radial interface.

\paragraph{Heuristic mechanism.}
Consider a nearly circular interface of radius \(R_0\) centred at \(x_0\), and
perturb it by an azimuthal mode of wavenumber \(k\). The Cahn--Hilliard part
relaxes interfacial curvature and therefore penalizes short-wavelength
perturbations; this is the mechanism responsible for coarsening when
\(\eta=0\). By contrast, for \(D=\nu\otimes\nu\) and a radial interface, \(\nu\)
is approximately aligned with the normal direction, so the traction
contribution to the conservative flux,
\(-\kappa\,m(\rho)(\nu\cdot\nabla\rho)\,\nu\) with \(\kappa=\varepsilon^{-1}\eta^2\),
points outward (since \(\rho\) decreases outward) and amplifies outward bulges
once the traction is large enough. A simple balance of scales makes the trend
explicit: along an interface of radius \(R\), an angular perturbation of
wavenumber \(k\) has tangential length scale \(R/k\); the fourth-order
Cahn--Hilliard regularization acts at the scale \(\varepsilon k^4/R^4\) and the
second-order traction term at the scale \(\kappa k^2/R^2\), so balancing them
gives
\begin{equation}\label{eq:fingering_scaling_heuristic}
k_{\mathrm{bal}}\sim R\sqrt{\frac{\kappa}{\varepsilon}}=\frac{R\eta}{\varepsilon},
\qquad
\lambda_{\mathrm{bal}}\sim\frac{2\pi R}{k_{\mathrm{bal}}}\sim\frac{\varepsilon}{\eta}.
\end{equation}
This argument is only heuristic; the precise growth rates depend on the
mobility, the interface profile, the radial geometry, the boundary, and
nonlinear saturation. We therefore use \eqref{eq:fingering_scaling_heuristic}
only as a qualitative guide: a traction threshold below which the radial
interface is stable, and, above it, a dominant wavenumber that increases and a
wavelength that decreases with \(\eta\).

\paragraph{Extraction of observables.}
To extract the protrusion count we threshold the phase field at level
\(\rho=\tfrac12\) and extract the outer interface as a polar graph
\(r=r_h(\theta)\) by radial bisection from \(x_0\); if several crossings occur
on a ray, the outermost is used. The count \(N_h\) is the number of azimuthal
local maxima of \(r_h\) whose prominence exceeds
\(\max\{0.05(r_{\max}-r_{\min}),\,4h\}\), the absolute floor \(4h\) preventing
spurious detections on a nearly circular interface, and the characteristic
wavelength is \(\lambda_h:=2\pi\bar r_h/N_h\) with \(\bar r_h\) the angular mean
of \(r_h\). We sweep \(\eta\in\{0.40,0.45,\dots,1.00\}\) at
\(\varepsilon=10^{-2}\), \(h=1/256\), \(\tau=10^{-5}\), averaging over three
independent smoothed-noise perturbations of the initial interface of amplitude
\(\|\zeta\|_{L^\infty}\le10^{-3}\).

\paragraph{The protrusion count is a transient.}
The count is not a stationary quantity: protrusions appear at an onset time
that decreases with \(\eta\), and once formed they coarsen, as neighbouring
protrusions merge and thin filaments shed satellites. A single fixed
observation time therefore does not give a clean \(N_h(\eta)\): small-\(\eta\)
configurations are sampled before onset while large-\(\eta\) configurations are
already coarsening. To expose this we record the count at a sequence of
observation times within a single forward run for each \(\eta\). The resulting
family \(N_h(\eta,t)\) is shown in Figure~\ref{fig:fingering-scaling} (left). No
protrusions form for \(\eta\le0.40\) within the observation window, identifying
a threshold \(\eta_c\approx0.42\). Above it, the onset time decreases
monotonically with \(\eta\), from \(t_{\mathrm{onset}}\approx3.5\times10^{-3}\)
at \(\eta=0.45\) to \(5\times10^{-4}\) at \(\eta=1.0\)
(Table~\ref{tab:fingering}), while for each \(\eta\) the count rises to a peak
shortly after onset and then decreases through coarsening; for \(\eta=1.0\), for
instance, \(N_h\) falls from about \(12\) near \(t=5\times10^{-4}\) to \(5\) by
\(t=5\times10^{-3}\). Consequently the apparent \(N_h(\eta)\) at any fixed time
is non-monotone and misleading.

\paragraph{Near-onset count.}
The quantity that most closely reflects the linear most-unstable mode is the
near-onset peak count \(N_h^{\mathrm{peak}}(\eta):=\max_t N_h(\eta,t)\), shown
in Figure~\ref{fig:fingering-scaling} (right) and tabulated in
Table~\ref{tab:fingering}. It increases monotonically with \(\eta\): it equals
\(6\) over the whole band \(0.45\le\eta\le0.70\) just above threshold, then
rises to \(8\) at \(\eta=0.8\) and to roughly \(11\)--\(12\) at \(\eta=1.0\),
while the corresponding wavelength decreases from \(\lambda_h\approx0.29\) to
\(\lambda_h\approx0.14\). The plateau at \(6\) is genuine: each case in
\(0.45\le\eta\le0.70\) jumps directly from \(N_h=0\) to \(N_h=6\) at its onset,
so \(k_\ast=6\) is the mode selected throughout the near-threshold band, and
only for \(\eta\ge0.8\) does the selected mode shift upward. This is consistent
with the qualitative prediction of \eqref{eq:fingering_scaling_heuristic} that
the dominant wavenumber grows with \(\eta\); the observed decrease of
\(\lambda_h\) is in fact somewhat steeper than the simple linear-in-\(\eta\)
estimate, which we attribute to the near-threshold mode selection and to
nonlinear effects not captured by the scale balance.

\begin{table}[htp!]
\centering
\caption{Onset and near-onset multiplicity of protrusions versus the traction
strength \(\eta\), at \(\varepsilon=10^{-2}\), \(h=1/256\), \(\tau=10^{-5}\),
averaged over three random initial perturbations. \(t_{\mathrm{onset}}\) is the
first observation time at which protrusions are detected;
\(N_h^{\mathrm{peak}}=\max_t N_h(\eta,t)\) is the near-onset peak count and
\(\lambda_h\) the corresponding wavelength. A dash indicates that no protrusions
form within the observation window \(t\in[2.5\times10^{-4},5\times10^{-3}]\).}
\label{tab:fingering}
\begin{tabular}{c|c|c|c}
\hline
\(\eta\) & \(t_{\mathrm{onset}}\) & \(N_h^{\mathrm{peak}}\) & \(\lambda_h\) at peak\\
\hline
\(0.40\) & ---                & ---    & ---    \\
\(0.45\) & \(3.5\times10^{-3}\) & \(6.0\)  & \(0.286\)\\
\(0.50\) & \(2.5\times10^{-3}\) & \(6.0\)  & \(0.285\)\\
\(0.55\) & \(2.0\times10^{-3}\) & \(6.0\)  & \(0.284\)\\
\(0.60\) & \(1.5\times10^{-3}\) & \(6.0\)  & \(0.285\)\\
\(0.70\) & \(1.0\times10^{-3}\) & \(6.0\)  & \(0.285\)\\
\(0.80\) & \(7.5\times10^{-4}\) & \(8.0\)  & \(0.202\)\\
\(0.90\) & \(7.5\times10^{-4}\) & \(10.7\) & \(0.150\)\\
\(1.00\) & \(5.0\times10^{-4}\) & \(11.7\) & \(0.141\)\\
\hline
\end{tabular}
\end{table}

\begin{figure}[htp!]
\centering
\includegraphics[page=7,width=.45\textwidth]{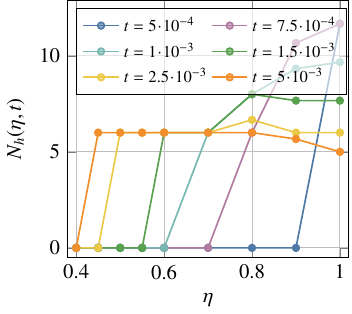}\hfill\includegraphics[page=8,width=.45\textwidth]{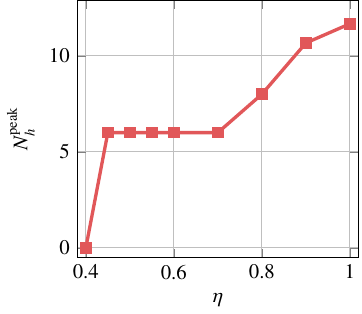}
\caption{Onset and multiplicity of protrusions. Left: the family
\(N_h(\eta,t)\) at several observation times. As \(t\) increases the onset
threshold in \(\eta\) moves left, while for large \(\eta\) the count decreases
through nonlinear coarsening, so a single fixed time gives a misleading
\(N_h(\eta)\). Right: the near-onset peak count
\(N_h^{\mathrm{peak}}(\eta)=\max_t N_h(\eta,t)\), the cleanest discrete proxy
for the linear most-unstable mode, which increases monotonically with \(\eta\)
above the threshold \(\eta_c\approx0.42\).}
\label{fig:fingering-scaling}
\end{figure}

\subsection{Manufactured-solution convergence study}
\label{subsec:convergence_study}

Theorem~\ref{thm:conv_scheme} establishes convergence of the fully discrete
scheme to a weak solution of the regularized problem as \(h,\tau\to0\), but
does not provide a rate. We complement it with a quantitative
manufactured-solution study, for which an exact smooth reference is available.

\paragraph{Setup.}
On \(\Omega=(0,1)^2\) we prescribe
\[
\rho_\ast(x,t):=\tfrac12+\tfrac14\cos(\pi x_1)\cos(\pi x_2)e^{-t},
\qquad
\mu_\ast:=\varepsilon^{-1}W'(\rho_\ast)-\varepsilon\Delta\rho_\ast ,
\]
insert \((\rho_\ast,\mu_\ast)\) into
\eqref{eq:reg_model_rho}--\eqref{eq:reg_model_mu}, and add the residual
\[
f:=\partial_t\rho_\ast-\Div\bigl(m(\rho_\ast)\nabla\mu_\ast\bigr)
-\kappa\Div\bigl(m(\rho_\ast)D(\nu)\nabla\rho_\ast\bigr)
\]
to the right-hand side of the discrete mass balance
\eqref{eq:scheme1_imex_num}, leaving the chemical-potential equation unchanged.
By construction \(f\) is smooth, has zero spatial mean, and respects the no-flux
boundary conditions, so that mass is conserved exactly at the discrete level.
We take a constant anisotropy direction \(\nu=e_1\), which renders the natural
total-flux boundary condition of the scheme consistent with \(\rho_\ast\). We use \(\varepsilon=5\times10^{-2}\), integrate to \(T=10^{-3}\), set
\(\eta=0.5\), and refine in space and time simultaneously with \(h_i=2^{-i}\),
\(\tau_i=C_\tau h_i^2\); this parabolic scaling keeps the temporal error below
the dominant spatial error throughout.

\paragraph{Errors and experimental orders.}
We report
\[
e^\rho_{L^\infty L^2}:=\|\rho_\ast-\rho_{h,\tau}\|_{L^\infty(0,T;L^2)},
\quad
e^\rho_{L^2 H^1}:=\|\rho_\ast-\rho_{h,\tau}\|_{L^2(0,T;H^1)},
\quad
e^\mu_{L^2 L^2}:=\|\mu_\ast-\mu_{h,\tau}\|_{L^2(0,T;L^2)},
\]
and \(\mathrm{EOC}_i:=\log_2(e_{i-1}/e_i)\). For a smooth manufactured
solution, interpolation and implicit-Euler heuristics suggest second-order
behaviour in \(e^\rho_{L^\infty L^2}\) under the above scaling, and at least
first-order behaviour in \(e^\rho_{L^2 H^1}\) and \(e^\mu_{L^2 L^2}\); these are
expectations for the smooth benchmark, not consequences of
Theorem~\ref{thm:conv_scheme}. Table~\ref{tab:eoc-combined} reports the combined space--time refinement.
The coarse levels are mildly pre-asymptotic, while the finest levels show the
expected textbook rates for mixed \(P_1\) finite elements on smooth data:
second-order convergence for \(\rho\) in \(L^\infty(0,T;L^2(\Omega))\),
first-order convergence for \(\rho\) in \(L^2(0,T;H^1(\Omega))\), and
second-order convergence for \(\mu\) in \(L^2(0,T;L^2(\Omega))\). Thus the
observed asymptotic rates are \(2/1/2\), in agreement with the standard
interpolation heuristic, and no degradation due to the traction term is
observed in this smooth benchmark.

\begin{table}[htp!]
\centering
\caption{Experimental order of convergence for the manufactured solution under
combined space--time refinement \(\tau_i=C_\tau h_i^2\), at \(T=10^{-3}\),
\(\varepsilon=5\times10^{-2}\), \(\eta=0.5\).}
\label{tab:eoc-combined}
\begin{tabular}{c|cc|cc|cc}
\hline
\(i\) & \(e^\rho_{L^\infty L^2}\) & EOC & \(e^\rho_{L^2 H^1}\) & EOC
    & \(e^\mu_{L^2 L^2}\) & EOC\\
\hline
\(3\) & \(1.983\times10^{-3}\) & ---  & \(3.100\times10^{-3}\) & ---
    & \(9.496\times10^{-4}\) & --- \\
\(4\) & \(6.545\times10^{-4}\) & \(1.60\) & \(1.712\times10^{-3}\) & \(0.86\)
    & \(2.484\times10^{-4}\) & \(1.93\)\\
\(5\) & \(1.814\times10^{-4}\) & \(1.85\) & \(8.526\times10^{-4}\) & \(1.01\)
    & \(6.328\times10^{-5}\) & \(1.97\)\\
\(6\) & \(4.788\times10^{-5}\) & \(1.92\) & \(4.312\times10^{-4}\) & \(0.98\)
    & \(1.631\times10^{-5}\) & \(1.96\)\\
\(7\) & \(1.209\times10^{-5}\) & \(1.99\) & \(2.155\times10^{-4}\) & \(1.00\)
    & \(4.098\times10^{-6}\) & \(1.99\)\\
\hline
\end{tabular}
\end{table}

Because the parabolic scaling couples \(h\) and \(\tau\), we additionally
isolate the two contributions. With \(\tau\) fixed at its smallest value,
\(\tau=7.62939\times10^{-7}\), and \(h\) varying, the
\(L^\infty L^2\)-error of \(\rho\) has EOCs
\(1.95\), \(1.98\), \(1.99\), and \(2.00\), confirming clean second-order
spatial convergence:
\[
e^\rho_{L^\infty L^2}
=
2.928\times10^{-3},\ 7.577\times10^{-4},\
1.924\times10^{-4},\ 4.827\times10^{-5},\
1.208\times10^{-5}.
\]
With \(h=1/128\) fixed and \(\tau\) varying over a decade, the error decreases
and then plateaus,
\[
e^\rho_{L^\infty L^2}
=
1.955\times10^{-5},\ 1.382\times10^{-5},\
1.362\times10^{-5},\ 1.247\times10^{-5},\
1.215\times10^{-5}
\quad(\tau=8,4,2,1,0.5\times10^{-4}),
\]
showing that the temporal error has already fallen below the fixed spatial
error. Together these one-sided studies confirm that the reported combined
rates are effectively spatial.

\begin{figure}[htp!]
\centering
\includegraphics[page=9,width=.75\textwidth]{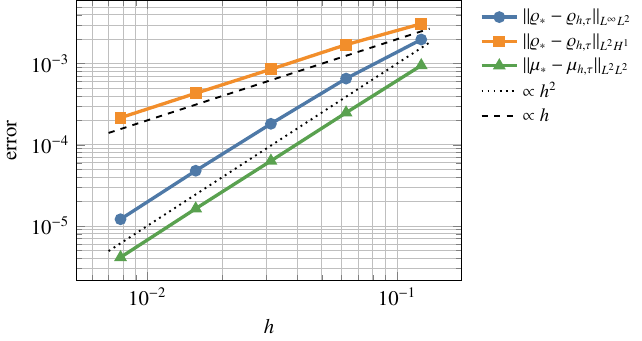}
\caption{Experimental convergence for the manufactured solution under the
parabolic refinement \(\tau_i=C_\tau h_i^2\), at \(T=10^{-3}\),
\(\varepsilon=5\times10^{-2}\), \(\eta=0.5\). The observed asymptotic rates
are second order for \(\rho\) in \(L^\infty L^2\), first order for \(\rho\)
in \(L^2H^1\), and second order for \(\mu\) in \(L^2L^2\), matching the
standard mixed-\(P_1\) interpolation heuristic; see
Table~\ref{tab:eoc-combined}.}
\label{fig:convergence}
\end{figure}

\paragraph{Interpretation.}
This study should be read as a numerical benchmark for the regularized, forced
problem and for the implementation of
\eqref{eq:scheme1_imex_num}--\eqref{eq:scheme2_imex_num}. It is consistent with
the compactness convergence theorem and with the textbook behaviour of a mixed
\(P_1\) convex-splitting scheme on smooth data: \(2/1/2\) convergence in the
three reported quantities. It does not constitute a proof of an error estimate,
and the rates reported here should not be confused with the rate-free
convergence statement of Theorem~\ref{thm:conv_scheme}.

\section{Conclusion and outlook}\label{sec:conclusion}

We have introduced a structure-preserving finite element discretization for a
Cahn--Hilliard equation with anisotropic traction flux, motivated by
mechanically biased cell interactions in digit-forming organoids. The additional traction flux is conservative but not generated by the
variational derivative of the Cahn--Hilliard energy, and therefore acts as a
non-variational perturbation of the mobility-weighted Cahn--Hilliard
gradient-flow dynamics.

For a regularized strictly positive mobility, we proposed a mixed
\(P_1\)--\(P_1\) scheme with a convex--concave splitting of the bulk energy and
an explicit treatment of the mobility. The method preserves the total mass
exactly and satisfies a discrete energy balance in which the traction term is
controlled as a perturbation of the Cahn--Hilliard dissipation. We proved
existence of fully discrete solutions and convergence, as \(h,\tau\to0\), to a
weak solution of the regularized problem.

The numerical experiments confirm the main structural and qualitative
properties of the method. Mass is conserved up to solver precision, and the
energy-balance residual verifies the discrete estimate at the algebraic level.
For small traction strength, the dynamics resemble classical Cahn--Hilliard
coarsening. For larger traction, the radial anisotropic flux destabilizes the
diffuse interface and produces persistent finger-like protrusions aligned with
the prescribed direction field. The fingering diagnostics identify a traction
threshold, a decreasing onset time, and an increasing near-onset protrusion
count as \(\eta\) grows. The manufactured-solution convergence study further
shows the expected textbook mixed-\(P_1\) rates.

Several questions remain open. First, the present analysis treats a
regularized mobility; passing to the genuinely degenerate mobility of the
original model at the fully discrete level would require estimates uniform in
the mobility regularization. Second, although the manufactured-solution study exhibits the expected
textbook rates in a smooth benchmark, the convergence theorem itself is
qualitative. Deriving rigorous optimal error estimates for smooth solutions,
including the explicit mobility and traction perturbation, remains an
interesting open problem. Third, a
sharp-interface or matched-asymptotic analysis could clarify the observed
traction-induced instability thresholds and wavelength selection. Finally, it
would be natural to couple the prescribed traction tensor to evolving
mechanical, orientational, or biochemical fields. The conservative finite
element framework developed here provides a basis for such extensions. \vspace{-.2cm}

\setlength{\bibsep}{0.2\baselineskip}
\bibliographystyle{model1b-num-names}
\bibliography{literature}

@article{Agosti2018,
	title        = {Error analysis of a finite element approximation of a degenerate {C}ahn--{H}illiard equation},
	author       = {Agosti, Abramo},
	year         = {2018},
	journal      = {ESAIM: Mathematical Modelling and Numerical Analysis},
	volume       = {52},
	number       = {3},
	pages        = {827--867}
}

@book{roubicek2005nonlinear,
  title={Nonlinear Partial Differential Equations with Applications},
  author={Roub{\'\i}{\v{c}}ek, Tom{\'a}{\v{s}}},
  year={2005},
  publisher={Springer}
}

@article{barrett_blowey_1999,
  title={Finite element approximation of the {C}ahn--{H}illiard equation with concentration dependent mobility},
  author={Barrett, John and Blowey, James},
  journal={Mathematics of Computation},
  volume={68},
  number={226},
  pages={487--517},
  year={1999}
}

@article{barrett1999finite,
	title        = {Finite element approximation of the {C}ahn--{H}illiard equation with degenerate mobility},
	author       = {Barrett, John W and Blowey, James F and Garcke, Harald},
	year         = {1999},
	journal      = {SIAM Journal on Numerical Analysis},
	publisher    = {SIAM},
	volume       = {37},
	number       = {1},
	pages        = {286--318}
}

@article{BarrettBloweyGarcke2001,
	title        = {On fully practical finite element approximations of degenerate {C}ahn--{H}illiard systems},
	author       = {Barrett, John W. and Blowey, James F. and Garcke, Harald},
	year         = {2001},
	journal      = {ESAIM: Mathematical Modelling and Numerical Analysis},
	volume       = {35},
	number       = {4},
	pages        = {713--748}
}

@article{BramblePasciakSteinbach2002,
  author  = {Bramble, James H. and Pasciak, Joseph E. and Steinbach, Olaf},
  title   = {On the stability of the {$L^2$} projection in {$H^1(\Omega)$}},
  journal = {Mathematics of Computation},
  volume  = {71},
  number  = {237},
  pages   = {147--156},
  year    = {2002},
  doi     = {10.1090/S0025-5718-01-01314-X}
}

@article{brunk2025analysis,
	title        = {Analysis and structure-preserving approximation of a {{C}ahn--{H}illiard-Forchheimer} system with solution-dependent mass and volume source},
	author       = {Brunk, Aaron and Fritz, Marvin},
	year         = {2025},
	journal      = {ESAIM Math. Model. Numer. Anal.},
	volume       = {59},
	number       = {6},
	pages        = {2991--3020}
}

@article{brunk2025ohta,
	title        = {Analysis and discretization of the {O}hta--{K}awasaki equation with forcing and degenerate mobility},
	author       = {Brunk, Aaron and Fritz, Marvin},
	year         = {2025},
	journal      = {Partial Differ. Equ. Appl.},
	publisher    = {Springer},
	volume       = {6},
	number       = {6},
	pages        = {Article 54}
}

@article{brunk2025structure,
	title        = {Structure-Preserving Approximation of the {{C}ahn--{H}illiard--Biot} System},
	author       = {Brunk, Aaron and Fritz, Marvin},
	year         = {2025},
	journal      = {Numer. Methods Partial Differ. Equ.},
	publisher    = {Wiley Online Library},
	volume       = {41},
	number       = {1},
	pages        = {e23159}
}

@article{brunk2026reviewthermodynamicstructuresstructurepreserving,
	title        = {Review of thermodynamic structures and structure-preserving discretisations of {C}ahn--{H}illiard-type models},
	author       = {Aaron Brunk and Marco F. P. ten Eikelder and Marvin Fritz and Dennis Höhn and Dennis Trautwein},
	year         = {2026},
    note={arXiv:2602.08791},
    journal={Preprint}
}

@article{CopettiElliott1992,
  title={Numerical analysis of the {C}ahn--{H}illiard equation with a logarithmic free energy},
  author={Copetti, Maria In{\^e}s Martins and Elliott, Charles M},
  journal={Numerische Mathematik},
  volume={63},
  number={1},
  pages={39--65},
  year={1992},
  publisher={Springer}
}

@book{evans2022partial,
	title        = {Partial Differential Equations},
	author       = {Evans, Lawrence C},
	year         = {2010},
	publisher    = {American Mathematical Society},
    series={Graduate Studies in Mathematics},
    volume=19
}

@article{FengProhl2004,
	title        = {Error analysis of a mixed finite element method for the {C}ahn--{H}illiard equation},
	author       = {Feng, Xiaobing and Prohl, Andreas},
	year         = {2004},
	journal      = {Numerische Mathematik},
	publisher    = {Springer},
	volume       = {99},
	number       = {1},
	pages        = {47--84}
}

@article{brunk2026high,
    author = {Brunk, Aaron and Fritz, Marvin},
    title ={High-order conforming finite elements for the {C}ahn--{H}illiard equation: {R}elative-energy stability and energy defects } ,
    journal = {arXiv},
    volume={2606.06719},
    year = 2026
}

@article{fritz2026global,
	title        = {Global weak solutions of a one-sided degenerate {C}ahn--{H}illiard model for traction-driven digit morphogenesis},
	author       = {Fritz, Marvin},
	year         = {2026},
	journal      = {arXiv},
    volume={2606.10793}
}

@article{garcke2022numerical,
	title        = {Numerical analysis for a {C}ahn--{H}illiard system modelling tumour growth with chemotaxis and active transport},
	author       = {Garcke, Harald and Trautwein, Dennis},
	year         = {2022},
	journal      = {Journal of Numerical Mathematics},
	publisher    = {De Gruyter},
	volume       = {30},
	number       = {4},
	pages        = {295--324}
}

@article{GarckeRumpfWeikard2001,
	title        = {The {C}ahn--{H}illiard equation with elasticity---finite element approximation and qualitative studies},
	author       = {Garcke, Harald and Rumpf, Martin and Weikard, Ulrich},
	year         = {2001},
	journal      = {Interfaces and Free Boundaries},
	volume       = {3},
	number       = {1},
	pages        = {101--118}
}

@article{kim2007numerical,
	title        = {A numerical method for the {C}ahn--{H}illiard equation with a variable mobility},
	author       = {Kim, Junseok},
	year         = {2007},
	journal      = {Communications in Nonlinear Science and Numerical Simulation},
	publisher    = {Elsevier},
	volume       = {12},
	number       = {8},
	pages        = {1560--1571}
}

@article{kim2009numerical,
	title        = {A numerical method for the ternary {C}ahn--{H}illiard system with a degenerate mobility},
	author       = {Kim, Junseok and Kang, Kyungkeun},
	year         = {2009},
	journal      = {Applied Numerical Mathematics},
	publisher    = {Elsevier},
	volume       = {59},
	number       = {5},
	pages        = {1029--1042}
}

@article{KondoMiura2010,
	title        = {Reaction-diffusion model as a framework for understanding biological pattern formation},
	author       = {Kondo, Shigeru and Miura, Takashi},
	year         = {2010},
	journal      = {Science},
	volume       = {329},
	number       = {5999},
	pages        = {1616--1620}
}

@article{Raspopovic2014,
	title        = {Digit patterning is controlled by a {Bmp}-{Sox9}-{Wnt} {T}uring network modulated by morphogen gradients},
	author       = {Raspopovic, Jelena and Marcon, Luciano and Russo, Luca and Sharpe, James},
	year         = {2014},
	journal      = {Science},
	volume       = {345},
	number       = {6196},
	pages        = {566--570},
	doi          = {10.1126/science.1252960}
}

@article{rathgeber2016firedrake,
	title        = {Firedrake: {A}utomating the finite element method by composing abstractions},
	author       = {Rathgeber, Florian and Ham, David A and Mitchell, Lawrence and et al},
	year         = {2016},
	journal      = {ACM Trans. Math. Softw.},
	publisher    = {ACM New York, NY, USA},
	volume       = {43},
	number       = {3},
	pages        = {Article 24, 1--27}
}

@article{Sheth2012,
	title        = {Hox genes regulate digit patterning by controlling the wavelength of a {T}uring-type mechanism},
	author       = {Sheth, Rushikesh and Marcon, Luciano and Bastida, Maria F. and Junco, Marta and Quintana, Laura and Dahn, Randall and Kmita, Marcin and Sharpe, James and Ros, Maria A.},
	year         = {2012},
	journal      = {Science},
	volume       = {338},
	number       = {6113},
	pages        = {1476--1480},
	doi          = {10.1126/science.1226804}
}

@article{TierraGuillenGonzalez2015,
	title        = {Numerical methods for solving the {C}ahn--{H}illiard equation and its applicability to related energy-based models},
	author       = {Tierra, Guillermo and Guill{\'e}n-Gonz{\'a}lez, Francisco},
	year         = {2015},
	journal      = {Archives of Computational Methods in Engineering},
	publisher    = {Springer},
	volume       = {22},
	number       = {2},
	pages        = {269--289}
}

@article{Tsutsumi2025,
  title={Mechanochemical instabilities drive digit morphogenesis in organoids},
  author={Tsutsumi, Rio and Diez, Antoine and Plunder, Steffen and Kimura, Ryuichi and Shinya, Oki and Takizawa, Kaori and Nakano, Rei and Akiyama, Haruhiko and Takada, Ritsuko and Takada, Shinji and others},
  journal={bioRxiv},
  volume={2025.08.31.673315},
  year={2025},
  publisher={Cold Spring Harbor Laboratory}
}

@article{Turing1952,
	title        = {The chemical basis of morphogenesis},
	author       = {Turing, Alan M.},
	year         = {1952},
	journal      = {Philosophical Transactions of the Royal Society of London. Series B, Biological Sciences},
	volume       = {237},
	number       = {641},
	pages        = {37--72}
}

\end{document}